\documentclass[10pt]{amsart}

\usepackage{amsmath}
\usepackage{amssymb}
\usepackage{amscd}
\usepackage{latexsym}
\usepackage[latin1]{inputenc}
\usepackage{enumerate}
\usepackage{xypic}

\usepackage{graphicx}
\usepackage{changebar,color}

\newtheorem{theorem}{Theorem}[section]
\newtheorem{corollary}[theorem]{Corollary}
\newtheorem{lemma}[theorem]{Lemma}

\newtheorem{remark}[theorem]{Remark}
\newtheorem{proposition}[theorem]{Proposition}

\newtheorem{definition}[theorem]{Definition}

\numberwithin{equation}{section}

\newcommand{\francois}[1]{{\color{blue} \sf $\spadesuit\spadesuit\spadesuit$ Fran\c{c}ois: [#1]}}

\def\Z{\mathbb{Z}}

\def\Q{\mathbb{Q}}

\def\C{\mathbb{C}}

\newcommand{\Spec}{{\rm Spec\, }}

\newcommand{\Pic}{{\rm Pic}}

\newcommand{\disc}{{\rm disc}}

\newcommand{\ra}{\rightarrow}

\date{\today}

\title[Zarhin's trick for $K3$ surfaces]{Birational boundedness for holomorphic symplectic varieties, Zarhin's trick for $K3$ surfaces, and the Tate conjecture}
\author{Fran\c cois Charles}
\address{Department of Mathematics, Massachusetts Institute of Technology, Cambridge, MA 02139-4307, USA; and CNRS, Laboratoire de Math\'ematiques d'Orsay, Universit\'e Paris-Sud, 91405 Orsay CEDEX, France}
\email{francois.charles@math.u-psud.fr}

\date{}
\begin{document}

\begin{abstract}
We investigate boundedness results for families of holomorphic symplectic varieties up to birational equivalence. We prove the analogue of Zarhin's trick for $K3$ surfaces by constructing big line bundles of low degree on certain moduli spaces of stable sheaves, and proving birational versions of Matsusaka's big theorem for holomorphic symplectic varieties. 

As a consequence of these results, we give a new geometric proof of the Tate conjecture for $K3$ surfaces over finite fields of characteristic at least $5$, and a simple proof of the Tate conjecture for $K3$ surfaces with Picard number at least $2$ over arbitrary finite fields -- including fields of characteristic $2$.
\end{abstract}

\maketitle

\section{Introduction}

\subsection{Main results}

The initial inspiration for the results of this text is the paper \cite{Zarhin74}. The main insight of Zarhin -- "Zarhin's trick" -- is the fact that if $A$ is an abelian variety over an arbitrary field $k$, then $(A\times \widehat A)^4$ admits a principal polarization. In particular, while the set of isomorphism classes of polarized abelian varieties of fixed dimension $g$ does not form a limited family if $g>1$, it does map naturally to the moduli space of principally polarized abelian varieties of dimension $8g$. As proved by Tate in \cite{Tate66}, this finiteness result directly implies the Tate conjecture for abelian varieties over finite fields.

The goal of this paper is to discuss an analogue of this circle of ideas for $K3$ surfaces, and explain applications to new proofs of the Tate conjecture for divisors on these surfaces. While some of our results are not new, one of the goal of this paper is to emphasize the role of certain geometric objects -- moduli spaces of twisted and untwisted sheaves -- regarding the existence of divisors on surfaces.

\bigskip

We first investigate Zarhin's trick for $K3$ surfaces and proceed in two steps. The first step is to construct big line bundles on moduli spaces of sheaves on $K3$ surfaces. A simplified version of our result, which is stated in detail in Theorem \ref{thm:Zarhins-trick}, is the following. 

\begin{theorem}
Let $k$ be a field, and let $d$ be a positive integer. Then there exists a positive integer $r$ such that for infinitely many positive integers $m$, if $(X, H)$ is a polarized $K3$ surface of degree $2md$ over $k$, then there exists a smooth, $4$-dimensional, projective moduli space $\mathcal M$ of stable sheaves on $X$ and a line bundle $L$ on $\mathcal M$ satisfying $c_1(L)^{4}=r$ and $q(L)>0$, where $q$ is the Beauville-Bogomolov form on $\mathcal M$.
\end{theorem}
We actually give a congruence condition on $m$ that ensures it satisfies the property above. 

The space $\mathcal M$ is a natural analogue of the product $(A\times \widehat A)^4$ appearing in Zarhin's trick. Indeed, at least over $\C$, it is an irreducible holomorphic symplectic variety -- that is, it is simply connected and its space of holomorphic $2$-forms is spanned by a single symplectic form. Furthermore, it is deformation-equivalent to the Hilbert scheme $X^{[2]}$ of $2$ points on $X$.

By an important theorem of Huybrechts \cite[3.10]{Huybrechts99}, either $L$ or its dual is big. This means that there exists a power of $L$ that induces a birational map from $\mathcal M$ onto a subvariety of projective space. Another major theme -- and the second step -- of this paper is investigating the extent to which a birational version of Matsusaka's big theorem holds in this setting. We formulate an optimistic possible result as a question.

\noindent\textbf{Question.} Let $\mathcal M$ be either a complex projective holomorphic symplectic variety or -- in positive characteristic -- a smooth projective moduli space of stable sheaves of dimension $2n$ on a $K3$ surface, and let $L$ be a big line bundle on $\mathcal M$ with $c_1(L)^{2n}=r$ and $q(L)>0$, where $q$ is the Beauville-Bogomolov form. Do there exist integers $N, d$ and $k$ depending only on $r$ and $n$ such that the complete linear system $|kL|$ induces a birational map from $\mathcal M$ onto a subvariety of degree at most $d$ of $\mathbb P^{m}$ with $m\leq N$ ? 

One could even ask whether the integer $k$ can be chosen independently of $r$. If one could control the singularities of a general member of $|L|$, this would follow in characteristic zero from Theorem 1.3 in \cite{HaconMcKernanXu14}.

We are not able to answer the question -- in positive characteristic, even the case where $L$ is assumed to be ample is not clear -- but give partial results in that direction. 
For $K3$ surfaces, geometric considerations following \cite{SaintDonat74} allow us to answer the question in any characteristic, up to replacing $L$ by a different line bundle $L'$ with self-intersection bounded in terms of $r$. This is Proposition \ref{proposition:birational-boundedness-for-K3}.

In higher dimension, we do not give a completely geometric proof, as we do not understand the geometry of linear systems well enough in that case -- see however related considerations in \cite{OGrady08}. Over the field of complex numbers, we can use the period map and the global Torelli theorem of \cite{Verbitsky13} to answer the question in Theorem \ref{thm:Matsusaka-birational}, again possibly changing the bundle $L$. This has the following consequence.

\begin{theorem}
Let $n$ and $r$ be two positive integers. Then there exists a scheme $S$ of finite type over $\C$, and a projective morphism $\mathcal X\ra S$ such that if $X$ is a complex irreducible holomorphic symplectic variety of dimension $2n$ and $L$ is a line bundle on $X$ with $c_1(L)^{2n}=r$ and $q(L)>0$, where $q$ is the Beauville-Bogomolov form, then there exists a complex point $s$ of $S$ such that $\mathcal X_s$ is birational to $X$.
\end{theorem}
In other words, the holomorphic symplectic varieties as above form a bounded family.

In positive characteristic, we deal with finite fields using a similar strategy -- in that case, the period map is replaced by the Kuga-Satake construction. This prevents us from dealing with characteristic $2$, as is customary. We also assume that the characteristic is different from $3$, but this might be an unimportant restriction. The finiteness result we obtain is Proposition \ref{proposition:weak-bb}.

\bigskip

The main application of our results is to the Tate conjecture for $K3$ surfaces over finite fields. In odd characteristic, it has been proved in \cite{Maulik12, tateinv}, and independently in \cite{MadapusiPera13}. The first of this proofs relies on results of Borcherds on the Picard group of Shimura varieties, while the second one uses construction of canonical models of certain Shimura varieties. We follow a different approach that first appeared in spirit in \cite{ArtinSwinnertonDyer73} and was discussed in \cite{LieblichMaulikSnowden14}. In the latter paper, it is proved that the Tate conjecture for $K3$ surfaces over a finite field $k$ is equivalent to the finiteness of the set of isomorphism classes of $K3$ surfaces over $k$. By refining the arguments of \cite{LieblichMaulikSnowden14}, we are able to use a version of this criterion, together with both our version of Zarhin's trick above and our birational boundedness results, and give a new proof of the following theorem.

\begin{theorem}\label{thm:main}
Let $X$ be a $K3$ surface over a finite field of characteristic at least $5$. Then $X$ satisfies the Tate conjecture.
\end{theorem}

It was the hope of the author that the techniques of this paper would be able to give a proof of the Tate conjecture for $K3$ surfaces over finite fields that might not rely on Kuga-Satake varieties. However, our proof of birational boundedness results for higher-dimensional holomorphic symplectic varieties in positive characteristic turned out to require this construction. The reason why it appears is that since the birational geometry of holomorphic symplectic varieties might be hard to control, it is very helpful to translate the problem in terms of abelian varieties where the birational geometry is trivial and Matsusaka's big theorem is known even in positive characteristic. It seems possible that further understanding of the underlying geometry might answer the question above along the lines of Proposition \ref{proposition:birational-boundedness-for-K3}.

Our last result, which is new only in characteristic $2$ but whose proof is in any case significantly simpler that all the other proofs of the Tate conjecture, should be seen as a modern rephrasing of the main result of \cite{ArtinSwinnertonDyer73} which dealt with elliptic $K3$ surfaces. It does not use the Kuga-Satake construction nor $p$-adic methods. It relies on the ideas of \cite{LieblichMaulikSnowden14}, together with birational boundedness for $K3$ surfaces.

\begin{theorem}\label{thm:picard-number-2}
Let $X$ be a $K3$ surface over a finite field of arbitrary characteristic. If the Picard number of $X$ is at least $2$, then $X$ satisfies the Tate conjecture.
\end{theorem}
It is perhaps interesting to notice that, after a finite extension of the base field, the hypothesis of the theorem above is satisfied as soon as $X$ satisfies the Tate conjecture -- if this holds, the Picard number of $X_{\overline k}$ should be even, see for instance \cite{deJongKatz00}. It would be very interesting to find a direct proof of this fact.

\bigskip

The paper is split in three parts, which are independent in some respect. In section 2, we prove a version of Zarhin's trick for $K3$ surfaces over arbitrary fields. This relies on the study of moduli spaces of stable sheaves on $K3$ surfaces and their cohomology, as initiated by Mukai. 

Section 3 is devoted to birational versions of Matsusaka's big theorem for holomorphic symplectic varieties, over $\C$ and finite fields. For $K3$ surfaces, we explain how to use results of Saint-Donat to prove the desired results, while in the other cases we need a finer analysis of some moduli spaces via period maps. This leads to technical complications in positive characteristic, which arise in particular due to the fact that there does not seem to be a satisfying definition of holomorphic symplectic varieties over arbitrary fields.

In section 4, we apply the aforementioned results to the Tate conjecture for $K3$ surfaces over finite fields. We follow the strategy of \cite{LieblichMaulikSnowden14} for the most part by using moduli spaces of twisted sheaves. We also discuss a simple proof of Theorem \ref{thm:picard-number-2}.

While the last section is arithmetic in nature, we hope that the first two might be of some interest even for complex geometers.

\subsection{A preliminary result}

We will need the following lifting result. It is certainly well-known to experts and follows easily from \cite{LieblichMaulik14}, see also \cite[Proposition A.1]{LieblichOlsson14}.

\begin{proposition}\label{proposition:lifting}
Let $X$ be a $K3$ surface over an algebraically closed field $k$ of positive characteristic, and let $L_1, \ldots, L_r$ be line bundles on $X$. Assume that $L_1$ is ample, and let $W$ be the ring of Witt vectors of $k$. If $r\leq 10$, there exists a finite flat morphism $S\ra \Spec W$, where $S$ is the spectrum of a discrete valuation ring, and a smooth projective relative $K3$ surface $\mathcal X\ra S$ such that 
\begin{enumerate}[(i)]
\item The special fiber of $\mathcal X\ra S$ is isomorphic to $X$;
\item The image of the specialization map 
$$\Pic(\mathcal X)\ra\Pic(X)$$
contains the classes of $L_1, \ldots, L_r$.
\end{enumerate}
\end{proposition}

\begin{proof}
If $X$ has finite height, the result follows from \cite[Corollary 4.2]{LieblichMaulik14}. In general, in the deformation space of $(X, L_1)$ over $k$, which has dimension $19$, the complement of the locus of surfaces of finite height has dimension $9$ by \cite{Artin74}, and the deformation space of $(X, L_1, \ldots, L_r)$ has codimension $r-1$. As a consequence, $(X, L_1, \ldots, L_r)$ is a specialization of a $K3$ surface with finite height, which allows us to conclude.
\end{proof}

\subsection{Acknowledgements}This work was started during the junior trimester on algebraic geometry at the Hausdorff Institute in Bonn. I would like to thank the Hausdorff Institute for excellent working conditions. I am very happy to thank Olivier Benoist, Daniel Huybrechts and Eyal Markman for many useful conversations and pointing out some errors in a first version of this text. Presenting an early version of these results during a workshop on $K3$ surfaces in Fudan University, Shanghai was very helpful, and I am happy to thank the organizers.

\section{A variant of Zarhin's trick for $K3$ surfaces}

\subsection{Moduli spaces of stable sheaves on $K3$ surfaces}

The goal of this section is to describe the geometry of moduli spaces of stable sheaves on $K3$ surfaces. Over the field of complex numbers, these results are well-known due to the work of Mukai, O'Grady and Yoshioka. We explain below how to extend them to arbitrary fields.

\bigskip

Mukai lattices were first defined in \cite{Mukai87}. Recently, they have been defined and studied in great generality in the paper \cite{LieblichOlsson14}. For our purposes, we only need very basic definitions which recall now. 

\begin{definition}\label{definition:Mukai-vector}
Let $k$ be a field with algebraic closure $\overline k$, and let $X$ be a $K3$ surface over $k$. Let $\ell$ be a prime number which is invertible in $k$.
\begin{enumerate}[(i)]
\item The \emph{$\ell$-adic Mukai lattice of $X$} is the free $\Z_{\ell}$-module 
$$\widetilde H(X_{\overline k}, \Z_{\ell}) := H^0(X_{\overline k}, \Z_{\ell})\oplus H^2(X_{\overline k}, \Z_{\ell}(1))\oplus H^4(X_{\overline k}, \Z_{\ell}(2))$$
endowed with the \emph{Mukai pairing} 
$$\langle (a,b,c), (a', b', c')\rangle = bb'-ac'-a'c.$$
\item Let $\omega$ be the numerical equivalence class of a closed point in $X_{\overline k}$. A \emph{Mukai vector} on $X$ is an element $v$ of 
$$N(X) := \Z\oplus\mathrm{NS}(X)\oplus \Z\omega.$$
We denote by $\mathrm{rk}(v)$ the \emph{rank} of $v$, that is, its first component, and by $c_1(v)$ its component in $\mathrm{Pic}(X)$. We identify a Mukai vector and its image in the $\ell$-adic Mukai lattice under the natural injection
$$N(X) \ra \widetilde H(X_{\overline k}, \Z_{\ell}).$$
\item Let $\mathcal F$ be a coherent sheaf on $X$. The \emph{Mukai vector of $\mathcal F$} is
$$v(\mathcal F):=  ch(\mathcal F)\sqrt{td_X} = \mathrm{rk}(\mathcal F)+c_1(\mathcal F) + (\chi(\mathcal F)-\mathrm{rk}(\mathcal F))\omega.$$
\end{enumerate}
\end{definition}

If $\mathcal F$ and $\mathcal G$ are two coherent sheaves on $X$, let $\chi(\mathcal F, \mathcal G)=\sum_i (-1)^i\mathrm{Ext}^i(\mathcal F, \mathcal G)$. If $\mathcal F$ is locally free, then $\chi(\mathcal F, \mathcal G)=\chi(\mathcal F^\vee\otimes\mathcal G)$. By the Riemann-Roch theorem, see \cite[Proposition 2.2]{Mukai87} we have the following.

\begin{proposition}\label{proposition:euler-characteristic}
Let $\mathcal F$ and $\mathcal G$ be two coherent sheaves on $X$. Then
$$\chi(\mathcal F, \mathcal G)=-v(\mathcal F).v(\mathcal G).$$
\end{proposition}

Given a Mukai vector $v$ and a polarization $H$ on $X$ -- that is, $H$ is an isomorphism class of ample line bundles on $X$ -- we denote by $\mathcal M_H(X, v)$ the moduli space of Gieseker-Maruyama $H$-stable sheaves $\mathcal F$ on $X$ such that $v(\mathcal F)=v$. This moduli space is well-defined as a quasi-projective scheme over $k$ in arbitrary characteristic by work of Langer \cite[Theorem 0.2]{Langer04b}. When $X$ is fixed, we will denote this moduli space by $\mathcal M_H(v)$.

\bigskip

Over the field of complex numbers, it is customary to require the polarization $H$ to be generic with respect to $v$. Over an arbitrary field, we will state our results under a stronger assumption on $v$.

\begin{definition}\label{def:fine-moduli-space}
Let $X$ be a $K3$ surface over a field $k$, and let $H$ be a polarization on $X$. We say that a Mukai vector $v$ on $X$ satisfies condition $(C)$ if 
\begin{enumerate}[(i)]
\item The vector $v$ is primitive, $\mathrm{rk}(v)>0$ and $v^2>0$;
\item Writing $v=\mathrm{rk}(v)+c_1(v)+\lambda\omega$, then 
$$\gcd(\mathrm{rk}(v), H.c_1(v), \lambda)=1;$$
\end{enumerate}
\end{definition}

The following theorem describes the geometry of the moduli space in arbitrary characteristic. We will refine some of these results below under additional assumptions.

\begin{theorem}\label{thm:general-properties-of-moduli}
Let $k$ be a field with algebraic closure $\overline k$, and let $X$ be a $K3$ surface over $k$. Let $v$ be a Mukai vector on $X$ satisfying condition $(C)$.
\begin{enumerate}[(i)]
\item The space $\mathcal M_H(v)$ is a smooth, projective, geometrically irreducible variety of dimension $v^2+2$ over $k$. It is deformation-equivalent to the Hilbert scheme $X^{[n]}$ parametrizing subschemes of dimension $0$ and length $n=\frac{v^2+2}{2}$ in $X$. It is endowed with a natural symplectic structure, i.e., the sheaf $\Omega^2_{X/k}$ has a global section which is everywhere non-degenerate.
\item If $k$ is the field $\C$ of complex numbers, then $\mathcal M_H(v)$ is an irreducible holomorphic symplectic variety. 
\item If $\ell$ is a prime number which is invertible in $k$, then the $\ell$-adic cohomology group $H^2(\mathcal M_H(v)_{\overline k}, \Z_{\ell}(1))$ is endowed with a canonical quadratic form $q$ satisfying the formula 
\begin{equation}\label{eq:Beauville-Bogomolov}
\forall \alpha\in H^2(\mathcal M_H(v)_{\overline k}, \Z_{\ell}(1)), (2n)!q(\alpha)^{n}=(n!)2^n\alpha^{2n}
\end{equation}
where $2n$ is the dimension of $\mathcal M_H(v)$.
\item There exists a canonical quadratic form on $NS(\mathcal M_H(v)_{\overline k})$. If $p>0$, this quadratic form has values in $\Z[1/p]$. If $p=0$, it has values in $\Z$. For any $\ell\neq p$, the first Chern class map 
$$c_1 : NS(\mathcal M_H(v)_{\overline k})\otimes\Z_{\ell}\ra H^2(\mathcal M_H(v)_{\overline k}, \Z_{\ell}(1))$$
is an isometry.

\item Let $\ell$ be a prime number which is invertible in $k$. Let $v^{\perp}$ be the orthogonal complement of $v$ in the $\ell$-adic Mukai lattice of $X$ -- by a slight abuse of notation, we do not make explicit the dependence in $\ell$. Then there exists a canonical, $\mathrm{Gal}(\overline k/k)$-equivariant, isomorphism
$$\theta_{v,\ell} : v^{\perp} \ra H^2(\mathcal M_H(v)_{\overline k}, \Z_{\ell}(1))$$
which is an isometry.
\item Assume that $k$ is either algebraically closed or finite. Let $v^{\perp}\cap N(X)$ be the orthogonal complement of $v$ in the lattice $N(X)$ of Mukai vectors on $X$. There exists an injective isometry 
$$\theta_v : v^{\perp}\cap N(X) \ra NS(\mathcal M_H(v))$$
such that the following diagram commutes 
\[
\xymatrix{v^{\perp}\cap N(X)\ar[rr]^{\theta_v}\ar[d]^{c_1} & &  NS(\mathcal M_H(v))\ar[d]^{c_1}\\
               v^{\perp}\ar[rr]^{\theta_{v,\ell}} & & H^2(\mathcal M_H(v)_{\overline k}, \Z_{\ell}(1))
}
\]
for any prime number $\ell$ as above. 
\item Assume that $k$ is either algebraically closed or finite. Then the cokernel of $\theta_v$ is a $p$-primary torsion group, where $p$ is the characteristic of $k$.
\end{enumerate}
\end{theorem}

\begin{proof}
Let $\mathcal F$ be a semistable torsion-free sheaf of Mukai vector $v$. Write $v=\mathrm{rk}(v)+c_1(v)+\lambda\omega$. By Proposition \ref{proposition:euler-characteristic} we have, for any integer $d$, 
$$\chi(\mathcal F(d))=\frac{\mathrm{rk}(v)}{2}d^2H^2+(H.c_1(v))d+\lambda+\mathrm{rk}(v).$$
Since $H.c_1(v), \mathrm{rk}(v)$ and $\lambda$ are relatively prime, this implies -- since $\mathcal F$ is semistable -- that if $\mathcal G$ is any coherent subsheaf of $\mathcal F$, then 
$$\frac{\chi(\mathcal F(d))}{\mathrm{rk}(\mathcal F)}>\frac{\chi(\mathcal G(d))}{\mathrm{rk}(\mathcal G)}$$
for any large enough integer $d$. Equivalently, $\mathcal F$ is stable. As a consequence, Theorem 0.2 of \cite{Langer04b} shows that $\mathcal M_H(v)$ is projective. By \cite[Corollary 0.2]{Mukai84}, it is smooth, endowed with a natural symplectic structure and, if non-empty, it is pure of dimension $v^2+2$. 

Let $n=\frac{v^2+2}{2}$. If $k=\C$, \cite[Main Theorem]{OGrady97} -- see also \cite[Theorem 8.1]{Yoshioka01} -- shows that the moduli space $\mathcal M_H(v)$ is an irreducible holomorphic symplectic variety birational to $X^{[n]}$. In particular, $\mathcal M_H(v)$ is not empty. By the main theorem of \cite{Huybrechts99}, $\mathcal M_H(v)$ is deformation-equivalent to $X^{[n]}$. The Lefschetz principle shows that these statements hold over any algebraically closed field of characteristic zero. 

Now assume $k$ is an arbitrary field. We want to show that $\mathcal M_H(v)$ is deformation-equivalent to $X^{[n]}$. For this, we can assume that $k$ is algebraically closed and, by the discussion above, that $k$ has positive characteristic. By Proposition \ref{proposition:lifting}, we can find a finite flat morphism $S\ra \Spec W$, where $W$ is the ring of Witt vectors of $k$, and a lifting $\mathcal X\ra S$ of $X$ over $S$ such that both $H$ and $c_1(v)$ lift to $\mathcal X$. As a consequence, the Mukai vector $v$ also lifts to $\mathcal X$.

Consider the relative moduli space $\mathcal M_H(\mathcal X, v)$, which exists by \cite[Theorem 0.2]{Langer04b}. By \cite[Theorem 1.17]{Mukai84}, it is smooth over $S$. Since its generic fiber is deformation-equivalent to $\mathcal X_{\eta}^{[n]}$, $\mathcal M_H(X,v)$ is deformation-equivalent to $X^{[n]}$. This shows $(i)$ and $(ii)$.

\bigskip

We now prove items $(iii)$ to $(v)$. For these, we can assume that $k$ is algebraically closed. First assume that $k=\C$. Then the Beauville-Bogomolov quadratic form on $H^2(\mathcal M_H(v), \Z_{\ell}(1))$ satisfies equation (\ref{eq:Beauville-Bogomolov}) by \cite[4.14]{OGrady08}. By the comparison theorem between singular and $\ell$-adic cohomology, this shows $(iii)$ for $k=\C$, hence for $k$ algebraically closed of characteristic zero. For general $k$, lifting as before by Proposition \ref{proposition:lifting}, the smooth base change theorem gives $(iii)$.

Let $\ell$ be a prime number invertible in $k$. The cycle class map gives an injection
$$c_1 : NS(\mathcal M_H(v))\ra H^2(\mathcal M_H(v), \Z_\ell(1)).$$
As a consequence, the Beauville-Bogomolov form on $H^2(\mathcal M_H(v), \Z_\ell(1))$ induces a quadratic form on $NS(\mathcal M_H(v)$ with values in $\mathcal \Z_\ell$. We denote it by $q$. We show that  $q$ actually takes values in $\Q$ and is independent of $\ell$. This will imply $(iv)$. If $k=\C$, this holds because the Beauville-Bogomolov form is actually defined on singular cohomology with integer coefficients. This shows that the result holds if $k$ has characteristic zero. Assume that $k$ has positive characteristic and choose a lifting $\mathcal X\ra S$ of $X, H$ and $v$ to characteristic zero as above.

Let $H_{\mathcal M}\in NS(\mathcal M_H(v))$ be an ample line bundle on $\mathcal M_H(X, v)$ that lifts to $\mathcal M_H(\mathcal X, v)$. Since $H_{\mathcal M}$ lifts to characteristic zero, the argument above shows that $q(H_{\mathcal M})$ is an integer independent of $\ell$.

Let $\overline\eta$ be a generic geometric point of $S$. By \cite[Th\'eor\`eme 5 and end of p.775]{Beauville83}, there exists a rational number $\lambda$, independent of $\ell$, such that for any $\alpha\in H^2(\mathcal M_H(\mathcal X_{\overline\eta}, v), \Z_\ell(1))$ such that $\alpha\cup H_{\mathcal M}^{2n-1}=0$, we have 
\begin{equation}\label{eq:alternate-BB}
q(\alpha)=\lambda\alpha^2 \cup H_{\mathcal M}^{2n-2}.
\end{equation}
Indeed, this is true over $\C$ by the result of Beauville quoted above, and thus holds over any algebraically closed field of characteristic zero. Furthermore, the same formula holds for $\mathcal M_H(v)$ by the smooth base change theorem. This readily implies that the quadratic form $q$ on $NS(\mathcal M_H(v))$ takes values in $\Q$ and is independent of $\ell$.  This proves $(iv)$.

Over the field of complex numbers, the map $\theta_{v, \ell}$ is defined on the level of singular cohomology with coefficients in $\Z$ in \cite[5.14]{Mukai87b} and $(v)$ holds by the main theorem of \cite{OGrady97}. By the same arguments as above, it holds over an arbitrary algebraically closed field. Since the morphism is canonical, it is Galois-equivariant.

\bigskip

We now prove $(vi)$. The map $\theta_{v,\ell}$ is induced by an algebraic correspondence with coefficients in $\Q$, see again \cite[5.14]{Mukai87b}. As a consequence, it induces a map 
$$\theta_{v} :  (v^{\perp}\cap N(X))\otimes\Q\ra NS(\mathcal M_H(v))\otimes\Q.$$
This map is clearly compatible with $\theta_{v, \ell}$ via the cycle class map. By the definition of the quadratic forms involved, this implies that $\theta_{v}$ is an injective isometry.

We claim that $\theta_{v, \overline k}$ is defined over $\Z$, that is, that it sends $v^{\perp}\cap N(X_{\overline k})$ to $NS(\mathcal M_H(v)_{\overline k})$. If $k$ has characteristic zero, this is due to the fact that $\theta_{v, \ell}$ is defined over $\Z_\ell$ for any prime number $\ell$ and that the cokernel of the cycle class map 
$$c_1 : NS(\mathcal M_H(v)_{\overline k})\ra H^2(\mathcal M_H(v)_{\overline k}, \Z_\ell(1))$$
has no torsion.

To show that $\theta_{v, \overline k}$ is defined over $\Z$ for arbitrary $k$, we lift once again to characteristic zero. Given $\alpha\in N(X_{\overline k})$, we can lift $X$, $H$, $v$ and $\alpha$ to characteristic zero by Theorem \ref{proposition:lifting} and apply the claim to the generic fiber. This shows $(vi)$ if $k$ is algebraically closed.

\bigskip

Assume now that $k$ is a finite field. Since $H^1(\mathcal M_H(v)_{\overline k}, \Z_\ell)=0$ -- as follows from $(ii)$ if $k$ is the field of complex numbers and from a lifting argument in general -- the Hochschild-Serre spectral sequence shows that the map 
$$H^2(\mathcal M_H(v), \Z_{\ell}(1))\ra H^2(\mathcal M_H(v)_{\overline k}, \Z_{\ell}(1))^{\mathrm{Gal}(\overline k/k)}$$
is an isomorphism, where the left-hand side denotes continuous étale cohomology. Furthermore, a classical argument involving the Kummer exact sequence shows that the cycle class map
$$\mathrm{Pic}(\mathcal M_H(v))\ra H^2(\mathcal M_H(v), \Z_{\ell}(1))$$
has a torsion-free cokernel. In other words, the cokernel of the cycle class map
$$\mathrm{Pic}(\mathcal M_H(v))\ra H^2(\mathcal M_H(v)_{\overline k}, \Z_{\ell}(1))^{\mathrm{Gal}(\overline k/k)}$$
is torsion-free.

Now let $\alpha\in N(X)\cap v^{\perp}$. Since $\theta_{v,\ell}$ is Galois-equivariant, the element $\theta_{v,\ell}(\alpha)\in H^2(\mathcal M_H(v)_{\overline k}, \Z_\ell(1))$ is Galois-invariant. Furthermore, by the argument above, some multiple of $\alpha$ is the Chern class of an element of $\mathrm{Pic}(\mathcal M_H(v))$. This shows that $\alpha$ belongs to $NS(\mathcal M_H(v))$.

\bigskip

Let $k$ be as in $(vii)$. To show the result, we need to show that if $\ell$ is any prime number invertible in $k$, then 
$$\theta_v\otimes\Z_\ell : (v^{\perp}\cap N(X))\otimes\Z_\ell\ra NS(\mathcal M_H(v))\otimes\Z_\ell$$
is an isomorphism. We already know that $\theta_v\otimes\Z_\ell$ is injective. Furthermore, in the commutative diagram 
\[
\xymatrix{(v^{\perp}\cap N(X))\otimes\Z_\ell\ar[rr]^{\theta_v\otimes\Z_\ell}\ar[d]^{c_1} & &  NS(\mathcal M_H(v))\otimes\Z_\ell\ar[d]^{c_1}\\
               v^{\perp}\ar[rr]^{\theta_{v,\ell}} & & H^2(\mathcal M_H(v)_{\overline k}, \Z_{\ell}(1))
}
\]
the lower horizontal map is an isomorphism and the cokernel of the two vertical maps are torsion-free as shown above. This implies that the cokernel of $\theta_v\otimes\Z_\ell$ is torsion-free. To show that $\theta_v\otimes\Z_\ell$ is an isomorphism, we now have to show that 
$$\theta_v\otimes\Q_\ell : (v^{\perp}\cap N(X))\otimes\Q_\ell\ra NS(\mathcal M_H(v))\otimes\Q_\ell$$
is an isomorphism. Let $H_{\mathcal M}$ be the homological equivalence class of an ample divisor on $\mathcal M_H(v)$ and consider the composition $\phi_\ell$
\[
\xymatrix{(v^{\perp}\cap N(X))\otimes\Q_\ell\ar[r]^{\theta_v} & NS(\mathcal M_H(v))\otimes\Q_\ell\ar[r]^{c_1} & H^2(\mathcal M_H(v)_{\overline k}, \Q_{\ell}(1))\ar[d]^{\cup H_{\mathcal M}^{2n-2}}\\
v^{\perp}\otimes\Q_\ell & & H^{4n-2}(\mathcal M_H(v)_{\overline k}, \Q_{\ell}(1))\ar[ll]^{\theta_{v, \ell}^{\vee}}
}
\]
where 
$$\theta_{v, \ell}^{\vee} : H^{2n-2}(\mathcal M_H(v), \Z_{\ell}(1))\ra v^{\perp}$$
is the Poincar\'e dual of $\theta_{v, \ell}$. Then $\phi_\ell$ is injective by the hard Lefschetz theorem, $(v)$ and $(vi)$. Furthermore, it sends $(v^{\perp}\cap N(X))\otimes\Q_\ell$ into itself since it is induced by an algebraic correspondence and is Galois-equivariant. As a consequence, it induces an automorphism of $(v^{\perp}\cap N(X))\otimes\Q_\ell$. Furthermore, by the same argument, the composition 
\[
\xymatrix{NS(\mathcal M_H(v))\otimes\Q_\ell\ar[r] & H^2(\mathcal M_H(v)_{\overline k}, \Q_{\ell}(1))\ar[r] & H^{4n-2}(\mathcal M_H(v)_{\overline k}, \Q_{\ell}(1))\ar[r] & v^{\perp_{\ell}}\otimes\Q_\ell
}
\]
is injective, and maps into $(v^{\perp}\cap N(X))\otimes\Q_\ell$. This implies that $\theta_v\otimes\Q_\ell$ is surjective and concludes the proof.
\end{proof}

Under suitable assumptions on the characteristic of the base field and the $K3$ surface $X$, we can also both describe the de Rham cohomology groups of $\mathcal M_H(v)$ and extend the description of the N\'eron-Severi group of $\mathcal M_H(v)$ to some non-algebraically closed fields.
\begin{proposition}\label{proposition:de-Rham-for-MH}
Let $k$ be an algebraically closed field of characteristic $p>0$ and let $(X, H)$ be a polarized $K3$ surface over $k$. Let $v$ be a Mukai vector on $X$ satisfying condition $(C)$. Let $W$ be the ring of Witt vectors of $k$. 

Assume that the triple $(X, v, H)$ lifts to a projective $K3$ surface over $W$. Then the Hodge to de Rham spectral sequence of $\mathcal M_H(v)$ degenerates at $E_1$ if $p>v^2+2$. In general, the Hodge numbers $h^{p,q}=H^q(\mathcal M_H(v), \Omega^p_{X/k})$ of $\mathcal M_H(v)$ satisfy the equalities
\begin{enumerate}
\item $h^{1, 0}=h^{0,1}=0$ if $p>2$;
\item $h^{2,0}=h^{0,2}=1$ and $h^{1,1}=21$ if $p>3$.
\end{enumerate}
\end{proposition}

\begin{proof}
As above, a projective lift of $(X, v, H)$ to $W$ induces a projective lift $\mathcal M_H(\mathcal X, v)$ of $\mathcal M_H(v)$ to $W$. By the main result of \cite{DeligneIllusie87}, and since $p>v^2+2=\mathrm{dim}\,\mathcal M_H(v)$, this implies that the Hodge to de Rham spectral sequence of $\mathcal M_H(v)$ degenerates at $E_1$.

If $p>2$, the result of \cite{DeligneIllusie87} still shows that $E_1^{p,q}=E_{\infty}^{p,q}$ if $p+q=1$, and if $p+q=2$ if $p>3$.

The statement regarding the Hodge numbers can be rephrased as saying that the Hodge numbers of $\mathcal M_H(v)$ are the same as those of $\mathcal M_H(\mathcal X_{\eta}, v)$, where $\mathcal X_{\eta}$ is the generic fiber of $\mathcal X$ over $W$. By the universal coefficient theorem for crystalline cohomology, it is enough to show that the second and third crystalline cohomology groups of $\mathcal M_H(v)$ are torsion-free. By the integral comparison theorem of Fontaine-Messing \cite{FontaineMessing87}, this is the case as soon as the corresponding $p$-adic cohomology groups of $\mathcal M_H(\mathcal X_{\eta}, v)$ are torsion-free. Since $\mathcal M_H(\mathcal X_{\eta}, v)$ is deformation-equivalent to a Hilbert scheme of points on a $K3$ surface in characteristic zero, it suffices to show that if $S$ is any projective complex $K3$ surface, then $H^2(S^{[n]}, \Z)$ and $H^3(S^{[n]}, \Z)$ are both torsion-free. 

The fact that $H^2(S^{[n]}, \Z)$ is torsion-free is proved in \cite[Remarque after Proposition 6]{Beauville83}. We now show that $H^3(S^{[n]}, \Z)=0$.

Following \cite{Beauville83}, let $S^{(n)}$ be the $n$-fold symmetric product of $S$, and let $\epsilon : S^{[n]}\ra S^{(n)}$ be the Hilbert-Chow morphism. Let $\pi : S^n\ra S^{(n)}$ be the natural map. Let $D$ be the diagonal in $S^{(n)}$, that is, the locus of elements $x_1+\ldots+ x_n$ such that $x_i=x_j$ for some $i\neq j$, and let $D_\ast$ be the open subset of $D$ consisting of zero-cycles of the form $2x_1+\ldots+ x_{n-1}$ where the $x_i$ are all distinct. We define $S^{(n)}_\ast =S^{(n)}\setminus (D\setminus D_\ast)$, $S^{[n]}_*=\epsilon^{-1}(S^{(n)}_\ast)$ and $S^{n}_*=\pi^{-1}(S^{(n)}_\ast)$. Then by \cite[Section 6]{Beauville83}, $S^{[n]}\setminus S^{[n]}_\ast$ has codimension $2$ in $S^{[n]}$ and $S^{[n]}$ is the quotient by the symmetric group of the blow-up of $S^n_\ast$ along the diagonal $\Delta_*=\pi^{-1}(D_\ast)$.

From the description above, it is straightforward to check that $H^3(S^{[n]}_\ast, \Z)=0$. Since $S^{[n]}\setminus S^{[n]}_\ast$ has codimension $2$ in $S^{[n]}$, the restriction morphism 
$$H^3(S^{[n]}, \Z)\ra H^3(S^{[n]}_\ast, \Z)$$
is injective, which shows the result.
\end{proof}

\begin{corollary}\label{corollary:lifting-to-W}
Let $k$ be a field of characteristic $p>2$ that is either finite or algebraically closed and let $(X, H)$ be a polarized $K3$ surface over $k$. Let $v$ be a Mukai vector on $X$ satisfying condition $(C)$. Let $W$ be the ring of Witt vectors of $k$. 

Assume that the triple $(X, v, H)$ lifts to a projective $K3$ surface $\mathcal X$ over $W$. Let $\eta$ be the generic point of $\mathrm{Spec}\, W$. 

Then there exists an isometry
$$\theta_{v, \eta} : v^{\perp}\cap N(\mathcal X_{\eta})\ra NS(\mathcal M_H(\mathcal X_{\eta}, v))$$
such that the following diagram commutes
\[
\xymatrix{
v^{\perp}\cap N(\mathcal X_{\eta}) \ar[r]^{\theta_{v, \eta}} \ar[d] & NS(\mathcal M_H(\mathcal X_{\eta}, v))\ar[d]\\
v^{\perp}\cap N(X) \ar[r]^{\theta_{v}} & NS(\mathcal M_H(X, v))
}
\]
where the vertical maps are the specialization maps.
\end{corollary}

\begin{proof}
By the proof of $(vi)$ in Theorem \ref{thm:general-properties-of-moduli}, we know that there exists an isometry
$$\theta_{v, \eta} : (v^{\perp}\cap N(\mathcal X_{\eta}))\otimes\Q\ra NS(\mathcal M_H(\mathcal X_{\eta}, v))\otimes\Q$$
making the analog of the diagram above commute. We need to show that $\theta_{v, \eta}$ sends $v^{\perp}\cap N(\mathcal X_{\eta})$ to $NS(\mathcal M_H(\mathcal X_{\eta}, v))$. 

By \cite{ElsenhansJahnel11} -- which is stated over $\Z_p$ but whose proof extends verbatim to our setting over $W$ -- the equality $H^1(\mathcal M_H(v), \mathcal O_{\mathcal M_H(v)})=0$ proved in Proposition \ref{proposition:de-Rham-for-MH} implies that the cokernel of the specialization map 
$$NS(\mathcal M_H(\mathcal X_{\eta}, v))\ra NS(\mathcal M_H(X, v))$$
is torsion-free, which shows the result.
\end{proof}

\begin{corollary}\label{corollary:comparison-of-discriminants}
Let $k$ be an algebraically closed field of characteristic $p$, and let $X$ be a $K3$ surface over $k$. Let $H$ be a polarization on $X$, and let $v$ be a Mukai vector on $X$ that satisfies condition $(C)$ of Definition \ref{def:fine-moduli-space}. If $p>0$, we can find nonnegative integers $\lambda$ and $t$ such that $0<\lambda\leq v^2$ and 
$$|\mathrm{disc}(NS(X))| = p^t\lambda |\mathrm{disc}(NS(\mathcal M_H(v)))|.$$
If $p=0$, then
$$|\mathrm{disc}(NS(X))| \leq v^2|\mathrm{disc}(NS(\mathcal M_H(v)))|.$$

\end{corollary}

\begin{proof}
We treat the case where $p>0$. By Theorem \ref{thm:general-properties-of-moduli}, $(vi)$ and $(vii)$, we have an injective isometry
$$\theta_v : v^{\perp}\cap N(X)\ra NS(\mathcal M_H(v)),$$
where $v^{\perp}$ is the orthogonal of $v$ in $N(X)=\Z\oplus NS(X)\oplus \Z\omega$. The cokernel of the map above is a $p$-primary torsion group. By \cite[Lemma 2.1.1]{LieblichMaulikSnowden14}, we have 
$$|\mathrm{disc}(v^{\perp}\cap N(X))| = p^r |\mathrm{disc}(NS(\mathcal M_H(v)))|$$
for some nonnegative integer $r$. Furthermore, we have a natural injection of lattices with torsion cokernel
$$\Z v\oplus (v^{\perp}\cap N(X))\hookrightarrow N(X).$$
Since the discriminant of $\Z v\oplus v^{\perp}$ is $v^2\mathrm{disc}(v^{\perp}\cap N(X))$, this implies by the same argument that 
$$|\mathrm{disc}(N(X))|\leq v^2|\mathrm{disc}(v^{\perp}\cap N(X))|.$$
Finally, since as a lattice, $N(X)\simeq NS(X)\oplus U$, where $U$ is the hyperbolic plane, we get the result.
\end{proof}
%
%\begin{remark}
%In the situation of the corollary above, if $p>3$, it can be shown using Proposition \ref{proposition:de-Rham-for-MH}, that the map $\theta_v : N(X)\cap v^{\perp}\ra NS(\mathcal M_H(v))$ is an isomorphism -- using an argument similar to Lemma \ref{lemma:comparison-Picard-groups}. As a consequence, one can remove the power of $p$ in the Corollary above.
%\end{remark}

\subsection{Low-degree line bundles on moduli spaces of stable sheaves on $K3$ surfaces}

If $n$ is an integer, let $\Lambda_{2n}$ denote the lattice 
$$\Lambda_{2n} = \langle 2n \rangle\oplus U,$$
where $U$ is the hyperbolic plane.

\begin{proposition}\label{proposition:low-degree-elements-in-lattices}
Let $d$ be a positive integer, and let $\Lambda$ be a rank $2$ positive definite sublattice of $\Lambda_{2d}.$ There exists a positive integer $N$ and nonzero integers $a,b$ such that if $m$ is any positive integer satisfying
\begin{enumerate}[(i)]
\item $m=1[N]$,
\item $m$ is prime to $a$ and $b$, and both $a$ and $b$ are quadratic residues modulo $m$,
\end{enumerate}
then there exists a primitive embedding of $\Lambda$ into $\Lambda_{2md}$.
\end{proposition}

\begin{proof}
We use a result of Nikulin that describes primitive embeddings of even lattices. We describe here its content in our case.

Let $\Lambda$ be any rank $2$ positive-definite lattice. Fix a positive integer $n$. Let $A_{\Lambda}=\Lambda^*/\Lambda$ be the discriminant group of $\Lambda$. It is endowed with a natural quadratic form $q_{\Lambda}$ with values in $\Q/2\Z$. Similarly, let $A_{2n}=\Z/2n\Z$ be the discriminant group of the even lattice $\Lambda_{2n}$, and let $q_{2n}$ be the natural quadratic form on $A_{2n}$. Then $q_{2n}(1)=\frac{1}{2n}\in\Q/2\Z$.

It is proved in \cite[Proposition 1.15.1]{Nikulin79} that primitive embeddings of $\Lambda$ into $\Lambda_{2n}$ are in one-to-one correspondence with the tuples $(V, W, \gamma, t)$ where $V\subset A_{\Lambda}$ and $W\subset A_{2n}$ are subgroups, $\gamma : V\ra W$ is an isomorphism respecting the restrictions of $q_{\Lambda}$ and $q_{2n}$ to $V$ and $W$ respectively and $t$ is a positive integer such that the quadratic form 
$$(q_{\Lambda}\oplus (-q_{2n}))_{|\Gamma_{\gamma}^{\perp}}/\Gamma_{\gamma}$$
is isomorphic to the quadratic form on $\Z/2t\Z$ that sends $1$ to $\frac{1}{2t}$, where 
$$\Gamma_{\gamma}=\{(a,b)\in V\oplus W | \gamma(a)=b\}.$$
Note that $\Gamma_{\gamma}$ is a cyclic group as it can be identified to $W\subset \Z/2n\Z$.

\bigskip

In the setting of the Proposition, the primitive embedding of $\Lambda$ into $\Lambda_{2d}$ corresponds to a tuple $(H, H', \gamma, t)$. Let $m$ be a positive integer such that $m=1[4dt]$ and $m=1[|A_{\Lambda}|]$. Assume also that $m$ is prime to $(2n)!$. Note that multiplication by $m$ is the identity in $A_{\Lambda}$. We will make further assumptions on $m$ later on.

The map
\[
\xymatrix{A_{\Lambda}\oplus A_{2d}\ar[rr]^{\mathrm{Id}\oplus m\mathrm{Id}} & & A_{\Lambda}\oplus A_{2md}}
\]
is injective and respects the quadratic forms $q_{\Lambda}\oplus (-q_{2d})$ on the left and $q_{\Lambda}\oplus (-q_{2md})$ on the right. Indeed, since $m=1[4d]$, we have 
$$q_{2md}(m)=\frac{m^2}{2md}=\frac{m}{2d}=\frac{1}{2d}\in\Q/2\Z.$$

Let $W'$ be the image of $W$ in $A_{2nd}$, and let
$$\gamma' : V\ra W'$$
be the isometry induced by $\gamma$.

By assumption, the group $\Gamma_{\gamma}^{\perp}/\Gamma_{\gamma}$ is isomorphic to $\Z/2t\Z$. We can also find an element $(x_0, y_0)$ of $\Gamma_{\gamma}^{\perp}\subset A_{\Lambda}\oplus A_{2d}$ that maps to a generator of $\Gamma_{\gamma}^{\perp}/\Gamma_{\gamma}$ and such that 
$$q_{\lambda}(x_0)-\frac{y_0^2}{2d}=\frac{1}{2t}\in\Q/2\Z.$$

We see $y_0$ as an integer between $1$ and $2d$. Then we can consider $(x_0, y_0)$ as an element of $A_{\Lambda}\oplus A_{2md}$.

Let $(\alpha, \beta)$ be a generator of $\Gamma_{\gamma}$, where $\alpha\in A_{\Lambda}$ and $\beta\in A_{2d}$ is considered as an integer. Then $(\alpha, m\beta)$ is a generator of $\Gamma_{\gamma'}$. Let $b_{\lambda}, b_{2d}, b_{2md}$ be the bilinear forms with values in $\Q/\Z$ associated with $q_{\Lambda}, q_{2d}$ and $q_{2md}$ respectively. We have
$$b_{\Lambda}(x_0, \alpha)-b_{2d}(y_0, \beta)=b_{\Lambda}(x_0, \alpha)-\frac{y_0\beta}{2d}=0\in\Q/\Z$$
since $(x_0, y_0)\in\Gamma_{\gamma}^{\perp}$. This implies that
$$b_{\Lambda}(x_0, \alpha)-b_{2md}(y_0, m\beta)=b_{\Lambda}(x_0, \alpha)-\frac{my_0\beta}{2md}=0\in\Q/\Z,$$
which shows that $(x_0, y_0)$ belongs to $\Gamma_{\gamma'}^{\perp}$ in $A_{\Lambda}\oplus A_{2md}$.

By construction and since $m$ is prime to $y_0$, the order of $(x_0, y_0)$, seen as an element of $\Gamma_{\gamma'}^{\perp}$, in the group $\Gamma_{\gamma'}^{\perp}/(\Gamma_{\gamma'}^{\perp}\cap\mathrm{Im}(A_{\Lambda}\oplus A_{2d}))$, is exactly $m$. Furthermore, $m(x_0, y_0)$ is the image of the element $(mx_0, y_0)=(x_0, y_0)\in A_{\Lambda}\oplus A_{2d}$, which maps to a generator of $\Gamma_{\gamma}^{\perp}/\Gamma_{\gamma}$ by assumption.

\bigskip

Since $\Gamma_{\gamma}$ and $\Gamma_{\gamma'}$ are canonically isomorphic, the discussion above shows that the group $\Gamma_{\gamma'}^{\perp}/\Gamma_{\gamma'}$ is cyclic of order $2tm$, and has a generator $v$ such that 
$$q(v)=q_{\lambda}(x_0)-\frac{y_0^2}{2md}=\frac{1}{2t}+\frac{y_0^2}{2d}-\frac{y_0^2}{2md}=\frac{1}{2t}+\frac{y_0^2(m-1)}{2md}\in \Q/2\Z,$$
where $q$ is the natural quadratic form on $\Gamma_{\gamma'}^{\perp}/\Gamma_{\gamma'}$.

To show the result -- after adding conditions on $m$ according to the statement of the Proposition -- we need to find a generator $v'$ of $\Gamma_{\gamma'}^{\perp}/\Gamma_{\gamma'}$ such that $q(v')=\frac{1}{2mt}\in\Q/2\Z$. Writing $v'=\lambda v$, we need to find an integer $\lambda$ such that 
\begin{enumerate}[(i)]
\item $\lambda$ is prime to $2mt$;
\item $\lambda^2(\frac{1}{2t}+\frac{y_0^2(m-1)}{2md})-\frac{1}{2mt}\in 2\Z.$
\end{enumerate}
The second condition can be rephrased as the congruence
$$\lambda^2(md+ty_0^2(m-1))-d=0[4mdt].$$
From now on, we only consider integers $\lambda$ such that $\lambda=1[4dt]$. Since by assumption $m=1[4dt]$, this implies that the condition above is always satisfied modulo $4dt$. As a consequence, we only have to consider the condition modulo $m$, which becomes 
$$\lambda^2ty_0^2+d=0[m].$$
Note that, as above, $y_0$ is prime to $m$. Choosing $m$ such that both $-d$ and $t$ are both quadratic residues modulo $m$, this shows that we can find a suitable $\lambda$, and concludes the proof. 
\end{proof}

\begin{theorem}\label{thm:Zarhins-trick}
Let $k$ be a field, and let $d$ be a positive integer. Then there exists a positive integer $r$, a positive integer $N$ and nonzero integers $a,b$ such that if $(X, H)$ is a polarized $K3$ surface of degree $2md$ over $k$, where $m$ is any positive integer satisfying 
\begin{enumerate}[(i)]
\item $m=1[N]$,
\item $m$ is prime to $a$ and $b$, and both $a$ and $b$ are quadratic residues modulo $m$,
\end{enumerate}
then there exists a Mukai vector $v$ on $X$ satisfying condition $(C)$ such that 
\begin{enumerate}[(i)]
\item $c_1(v)$ is proportional to $c_1(H)$;
\item The moduli space $\mathcal M_H(v)$ has dimension $4$;
\item There exists a line bundle $L$ on $\mathcal M_H(v)$ satisfying $c_1(L)^{4}=r$ and $q(L)>0$.
\end{enumerate}

If $n$ is any integer not divisible by $2$ or $3$, we can assume that $q(L)$ is an integer prime to $n$, and that there exists an ample line bundle $A$ on $\mathcal M_H(v)$ such that $q(A)$ is an integer prime to $n$. If furthermore $k$ is algebraically closed or finite, the same result holds even when $n$ is divisible by $3$.

Finally, assume that $k$ has characteristic $p>2$, and let $W$ be the ring of Witt vectors of an algebraic closure $\overline k$ of $k$. If the pair $(X_{\overline k}, H)$ lifts to $W$, then we can assume that the triple $(\mathcal M_H(v)_{\overline k}, L, A)$ lifts to $W$. 
\end{theorem}

\begin{proof}
In the lattice $\Lambda_{2d}$, consider a positive-definite rank-$2$ sublattice $\Lambda$ containing elements $v$ and $w$ with $v^2=2$ and $v.w=1$. Let $l$ be an element of $\Lambda$ such that $l.v=0$, and let $r=\frac{(4)!}{(2!)2^2}(l^2)^2=3(l^2)^2$. Note that we can indeed choose $l$ so that $l^2$ is prime to $n$ if $n$ is odd.

By Proposition \ref{proposition:low-degree-elements-in-lattices}, we can find integers $N, a$ and $b$ as above such that if $m$ is any positive integer satisfying the conditions of the theorem, $\Lambda_{2md}$ contains $\Lambda$ as a primitive sublattice.

Let $X$ be any $K3$ surface over $k$ with an ample line bundle $H$ of self-intersection $2md$, with $m$ as above. Then the lattice $N(X)$ of Mukai vectors on $X$ contains the sublattice $\Z\oplus\Z H\oplus \Z\omega\simeq \Lambda_{2md}$. As a consequence, there exists an injection 
$$\Lambda\hookrightarrow \Lambda_{2md}\hookrightarrow N(X).$$
Seeing $v\in\Lambda$ as an element of $N(X)$, write $v=\mathrm{rk}(v)+c_1(v)+\lambda\omega$. By assumption there exists $w\in \Lambda\subset \Z\oplus\Z H\oplus \Z\omega$ such that $v.w=1$. In particular, the vector $v\in N(X)$ is primitive. Furthermore, we have 
$$gcd(\mathrm{rk}(v), c_1(v).H, \lambda)=1.$$

In particular, since $2$ divides $c_1(v).H$, $\mathrm{rk}(v)$ and $\lambda$ cannot both vanish. After composing the embedding of $\Lambda$ into $\Lambda_{2md}$ by a suitable automorphism of $\Lambda_{2md}$, we can assume that $\mathrm{rk}(v)>0$. Since $v^2>0$, this shows that $v$ satisfies condition $(C)$. As a consequence of Theorem \ref{thm:general-properties-of-moduli}, $(i)$, the moduli space $\mathcal M_H(v)$ has dimension $v^2+2=4$.

\bigskip

We first assume that $k$ is algebraically closed or finite. Let $n$ be an odd integer. Then the vector $l\in\Lambda\subset N(X)$ defined above lies in $v^{\perp}\cap N(X)$ by assumption. By Theorem \ref{thm:general-properties-of-moduli}, $(vi)$, we have an injection of lattices 
$$v^{\perp}\cap N(X)\hookrightarrow NS(\mathcal M_H(v)).$$
By $(iii)$ and $(iv)$ of the same theorem, the image of $l$ in $NS(\mathcal M_H(v))$ is the class of a line bundle $L$ on $\mathcal M_H(v)$ such that 
$$c_1(L)^{4}=3(l^2)^2=r.$$
The integer $q(L)$ is prime to $n$.

Let $A_0$ be an ample divisor on $\mathcal M_H(v)$. We can assume that $q(A_0)$ is an integer after raising to a sufficiently large $p$-th power. If $\lambda$ is large enough and $n$ divides $\lambda$, then $A=A_0^{\otimes \lambda}\otimes L$ is ample and $q(A)$ is congruent to $q(L)$ modulo $n$, which implies that $q(A)$ is prime to $n$.

If $k$ is an arbitrary field, the construction above provides a line bundle $L$ on $\mathcal M_H(v)_{\overline k}$ with Galois-invariant first Chern class. This implies that $L$ itself is Galois-invariant. 

Consider the exact sequence 
$$\mathrm{Pic}(\mathcal M_H(v))\ra \mathrm{Pic}(\mathcal M_H(v)_{\overline k})^{\mathrm{Gal}(\overline k/k)}\ra \mathrm{Br}(k)\ra \mathrm{Br}(\mathcal M_H(v)).$$
Since $\mathcal M_H(v)$ is deformation-equivalent to $X^{[2]}$, $c_4(\mathcal M_H(v))$ is a zero-cycle of degree $324=2^2\times 3^4$, see \cite[Remark 5.5]{EllingsrudGottscheLehn01}. Such a zero-cycle induces a map $Br(X)\ra Br(k)$ such that the composition $Br(k)\ra Br(X)\ra Br(k)$ is multiplication by $324$. In particular, the cokernel of the map $\mathrm{Pic}(\mathcal M_H(v))\ra \mathrm{Pic}(\mathcal M_H(v)_{\overline k})^{\mathrm{Gal}(\overline k/k)}$ is killed by multiplication by $324$. This shows that $L^{\otimes 324}$ satisfies the conclusion of the theorem.

\bigskip

Finally, assume that $k$ has characteristic $p>2$, and assume that the pair $(X_{\overline k}, H)$ lifts to a projective $K3$ surface $(\mathcal X, H)$ over $W$. Then since $c_1(v)$ is a multiple of $c_1(H)$, $v$ lifts to $\mathcal X$ as well. The result then follows directly from Corollary \ref{corollary:lifting-to-W}.
\end{proof}

\begin{remark}
Even over an arbitrary field, it is possible to ensure that $q(L)$ and $q(A)$ are prime to $3$ by considering a $6$-dimensional moduli space of sheaves: these have top Chern class of degree $3200=2^7\times 5^2$.
\end{remark}

\section{Finiteness results for holomorphic symplectic varieties}

Theorem \ref{thm:Zarhins-trick} was devoted to constructing irreducible holomorphic symplectic varieties of dimension $4$ -- or, in positive characteristic, reduction of such varieties -- together with a line bundle $L$ of low positive self-intersection. By a theorem of Huybrechts \cite[Corollary 3.10]{Huybrechts99}, either $L$ or its dual is big. We now investigate finiteness results for families of such varieties. 

In characteristic zero, we prove that given positive integers $n$ and $r$, the family of irreducible holomorphic varieties $X$ such that there exists a line bundle $L$ on $X$ with $c_1(L)^{2n}=r$ is birationally bounded. Unfortunately, our proof does not make explicit any of the natural constants involved. It relies on the global period map and the local Torelli theorem. 

Over finite fields of characteristic $p>3$, we show a finiteness result for N\'eron-Severi groups of such varieties. The proof relies on the Kuga-Satake construction as a replacement for the period map.

\subsection{The case of $K3$ surfaces}

Before dealing with higher-dimensional varieties below, we treat the much easier case of $K3$ surfaces. The following result is certainly well-known to experts. The proof is very close to arguments in \cite{SaintDonat74}, but since this paper assumes that the characteristic is odd, we make sure that the statement is correct in arbitrary characteristic.

\begin{proposition}\label{proposition:birational-boundedness-for-K3}
Let $r$ be a positive integer. Then there exist positive integers $N$ and $d$ such that if $k$ is any field, and if $X$ is a $K3$ surface over an algebraically closed field $k$ with a line bundle $L$ such that $L^2=r$, then there exists a line bundle $L'$ on $k$ with $h^0(X, L')\leq N$ such that the complete linear system $|L'|$ induces a birational map from $X$ to a subvariety of $\mathbb P|L'|$ of degree at most $d$.
\end{proposition}

\begin{proof}
In this proof, we will identify $L$ and $c_1(L)$. We can assume that $L$ is big and nef. Indeed, by \cite[1.10 and p.371]{Ogus79}, see also \cite[Chapter 8 par.2]{HuybrechtsK3book}, there exists a line bundle $L'$ with the same self-intersection as $L$ that is big and nef. In particular, we have $h^1(L)=0$ by \cite[Proposition 3.1]{HuybrechtsK3book}.

By the Riemann-Roch theorem, we have 
$$h^0(L)+h^0(L^{\vee})= \frac{r}{2}+2\geq 3.$$
Since $L$ is nef, this shows that $L$ is effective and $h^0(L)\geq 3$. 

We now rephrase the argument of \cite[Proposition 8.1]{SaintDonat74}. Let $F$ be the fixed part of the linear system $|L|$. Assume that $F\neq 0$. Let $M=L\otimes \mathcal O(-F)$. Then $M^2\geq 0$. By \cite[2.6 and 2.7.4]{SaintDonat74}, which does not make use of any hypothesis on the characteristic, either $M^2>0$ or we can find an irreducible curve $E$ on $X$ with arithmetic genus $1$ such that $M=\mathcal O(mE)$ with $m=h^0(L)-1\geq 2$. Furthermore, we can then find an irreducible rational curve $\Gamma$ in $F$ such that $F.E=1$. 

Let $L'=M$ if $M^2>0$ and $L'=mE+\Gamma$ in the other case. Write $L=L'+\Delta$. Assume that $\Delta\neq 0$. Then by the Hodge index theorem, see \cite[Chapter 1, Remark 2.4 (iii)]{HuybrechtsK3book}, $L'.\Delta>0$. Now $h^0(L')=h^0(L)= \frac{L^2}{2}+2$ and $h^1(L')=0$ by \cite[Lemma 2.2]{SaintDonat74}. This shows that $L'.L'= L.L$. In other words, we have 
$$2L'.\Delta+\Delta^2= 0.$$
However, since $L'.\Delta>0$, this implies that 
$$L.\Delta=L'.\Delta+\Delta^2<0$$
which contradicts the fact that $L$ is nef. This shows that $\Delta=0$, i.e., $L=L'$.

Now up to replacing $L$ by $2L$, it is readily seen that we can assume that $L$ has no fixed part. By \cite[Proposition 2.6]{SaintDonat74} -- see also \cite[Chapter 2, Remark 3.7 (ii)]{HuybrechtsK3book} -- we can write $L=\mathcal O(C)$ where $C$ is an irreducible curve on $X$. Furthermore, we have $h^1(L)=0$ by \cite[Proposition 3.1]{HuybrechtsK3book} again.

The discussion above readily implies that the image of the rational map $\phi_L$ from $X$ to $\mathbb P^{\frac{r}{2}+1}$ induced by the complete linear system $|L|$ has dimension $2$. Furthermore, $\phi_L$ is either birational or has generic degree $2$ onto its image. Indeed, the degree of the image of $\phi_L$ in projective space is at least $\frac{r}{2}$.

It is stated in the literature that $\phi_{2L}$ is birational onto its image. If the characteristic is odd, this follows from the existence of a smooth, irreducible global section of $L$ as in \cite{SaintDonat74}. In general, this is stated as "well-known" in \cite{ArtinSwinnertonDyer73}, after Lemma 5.17. We briefly give an argument that shows that $\phi_{4L}$ is birational onto its image.

The line bundle $L$ is big and net, and the fixed part of $L$ vanishes. By \cite[Chapter 2, Corollary 3.14 (i)]{HuybrechtsK3book}, $L$ is base point free. We assume that $\phi_L$ is of generic degree $2$ onto its image. Let $C$ be an irreducible curve belonging to $|L|$. It might not be possible to choose a reduced $C$ if the characteristic of $k$ is $2$. Let $C_{red}$ be the reduced corresponding curve. Then as cycles on $X$, we have $C=aC_{red}$ where $a=1$ or $a=2$ -- in characteristic $2$, it can happen that $a$ is necessarily $2$. By the adjunction formula -- which holds for non-reduced curves, see \cite[9.1.37]{Liu02}, the arithmetic genus of $C$ is $1+\frac{r}{2}$. 

The exact sequence
$$0\ra\mathcal (k-1)L\ra kL\ra \mathcal(kL)_{|C}\ra 0$$
together with the vanishing of $H^1(X, (k-1)L)$ shows that for any $k>0$ we have a surjection $H^0(X, kL)\ra H^0(C, kL)$. 

Now using the adjunction formula and Riemann-Roch (which both hold for non-reduced curves, see \cite[7.3.17 and 9.1.37]{Liu02}), the usual arguments show that the restriction of $4L$ to $C$ induces an application of degree $1$ onto its image. This shows that $\phi_{4L}$ is birational onto its image.

Since $h^1(X, 4L)=0$ and $4L$ is base point free, this shows that $\phi_{4L}$ induces a birational map from $X$ to a subvariety of degree $r$ of $\mathbb P^N$ with $N=\frac{r}{2}+1$.
\end{proof}

Since $K3$ surfaces are minimal, this shows that the surfaces $X$ as in the proposition above form a bounded family. In particular, we get the following result. 

\begin{corollary}\label{corollary:basic-finiteness}
Let $k$ be a finite field with algebraic closure $\overline k$, and let $r$ be a positive integer. Then there exist finitely many $\overline k$-isomorphism classes of $K3$ surfaces $X$ over $k$ such that there exists a line bundle $L$ on $X_{\overline k}$ with $L^2=r$.

If the characteristic of $k$ is odd, then there exist only finitely many isomorphism classes over $k$ of such $K3$ surfaces.
\end{corollary}

\begin{proof}
After replacing $k$ by a finite extension $K$ of fixed degree, we can assume that any line bundle on $X_{\overline k}$ is defined over $X$. Let $N, d$ and $L'$ be as in Proposition \ref{proposition:birational-boundedness-for-K3}. Then $L'$ is defined over $K$, and so is the image of $X$ under the rational map defined by the complete linear system associated to $L'$. 

The theory of Chow forms shows that there are only finitely many subvarieties of $\mathbb P^N_K$ of degree at most $d$. As a consequence, the number of $K$-birational classes of surfaces $X$ as in the statement is finite. Since $K3$ surfaces are minimal, this shows the first result. 

The second statement is a consequence of the first and of \cite[Proposition 2.4.1]{LieblichMaulikSnowden14}.
\end{proof}

\subsection{A birational version of Matsusaka's big theorem for holomorphic symplectic varieties}

The goal of this section is to prove the following result. It should be seen as a -- weak -- birational version of Matsusaka's big theorem. It will not be used in the proof of the Tate conjecture.

\begin{theorem}\label{thm:Matsusaka-birational}
Let $n$ and $r$ be two positive integers. Then we can find constants $k, N$ and $d$ such that if $X$ is a complex irreducible holomorphic symplectic variety of dimension $2n$ and $L$ is a line bundle on $X$ with $c_1(L)^{2n}=r$ and $q(L)>0$, where $q$ is the Beauville-Bogomolov form, then there exists a line bundle $L'$ with $c_1(L')^{2n}=r$ such that $h^0(X, kL')\leq N$ and such that the complete linear system $|kL'|$ induces a birational map from $X$ to a subvariety of $\mathbb P|kL'|$ of degree at most $d$.

In particular, there exists a scheme $S$ of finite type over $\C$, and a projective morphism $\mathcal X\ra S$ such that if $(X, L)$ is any pair as above, there exists a complex point $s$ of $S$ such that $\mathcal X_s$ is birational to $X$.
\end{theorem}

\begin{remark}\label{remark:easy-for-K3}
Our proof relies crucially on the existence of the global period map. It does not give an explicit value for the constants $k, N$ and $d$ -- in particular, they might depend on $r$. 
\end{remark}

We start with two lemmas. 

\begin{lemma}\label{lemma:finiteness-deformations}
There exist finitely many pairs $(X_1, L_1), \ldots, (X_s, L_s)$ where the $X_i$ are complex irreducible holomorphic symplectic manifolds of dimension $2n$ and $L_i$ is an ample line bundle on $X_i$ with $c_1(L_i)^{2n}=r$ such that if $(X, L)$ is a pair as in Theorem \ref{thm:Matsusaka-birational}, then either $(X, L)$ or $(X, L^{\otimes -1})$ is deformation-equivalent to $(X_i, L_i)$ for some $i$.
\end{lemma}

In the statement above and in the proof below, we are considering deformations of complex varieties over bases that are complex manifolds which are not necessarily projective.

\begin{proof}
Let $(X, L)$ be a pair as in Theorem \ref{thm:Matsusaka-birational}. By \cite[Theorem 3.11]{Huybrechts99} and \cite{Huybrechtserratum}, $X$ is projective. By the local Torelli theorem for $X$ \cite[Th\'eor\`eme 5]{Beauville83}, we can find a small deformation $(X', L')$ of the pair $(X, L)$ such that $\mathrm{Pic}(X')$ has rank $1$. By the aforementioned theorem of Huybrechts, $X'$ is projective, which implies, up to replacing $L$ by its dual, that $L'$ is ample.

Consider pairs $(X', L')$ where $X'$ is smooth projective of dimension $2n$, $K_{X'}=0$ and $L'$ is an ample line bundle with $c_(L)^{2n}=r$. By Koll\'ar-Matsusaka's refinement of Matsusaka's big theorem \cite{KollarMatsusaka83}, the family of such pairs $(X, L)$ is bounded. As a consequence, we can find finitely many pairs $(X_1, L_1), \ldots, (X_s, L_s)$ where the $X_i$ are complex irreducible holomorphic symplectic manifolds of dimension $2n$ and $L_i$ is an ample line bundle on $X_i$ with $c_1(L)^{2n}=r$ such that any pair $(X', L')$ as in the paragraph above is deformation-equivalent to one of the $(X_i, L_i)$.
\end{proof}

\begin{lemma}\label{lemma:boundedness-big}
Let $S$ be a noetherian scheme over $\C$, and let $\mathcal X\ra S$ be a projective morphism. Let $L$ be a line bundle over $\mathcal X$ such that for every complex point $s$ of $S$, the restriction $L_s$ of $L$ to $\mathcal X_s$ is big. Then there exists integers $k, N$ and $d$ such that for any complex point $s$ of $S$, $h^0(\mathcal X_s, kL_s)\leq N$ and the complete linear system $|kL_s|$ induces a birational map from $\mathcal X_s$ to a subvariety of $\mathbb P|kL_s|$ of degree at most $d$.
\end{lemma}

\begin{proof}
We use noetherian induction on $S$. It suffices to show that if $S$ is non-empty, there exists a non-empty open subset $U$ of $S$ and constants $k, N, d$ such that the conclusion of the lemma holds on $U$. 

Since $L_s$ is big for any complex point $s$ of $S$, Baire's theorem shows that if $\eta$ is any generic point of $S$, then $L_{\eta}$ is big. This readily shows the result.
\end{proof}

\begin{proof}[Proof of Theorem \ref{thm:Matsusaka-birational}]
By Lemma \ref{lemma:finiteness-deformations}, we can restrict our attention to the pairs $(X, L)$ that are deformation-equivalent to a given $(X_0, L_0)$ where $L_0$ is ample. Note that this implies that $L$ is big.

We denote by $\Lambda$ the lattice $H^2(X_0, \Z)$ endowed with its Beauville-Bogomolov form, and by $l\in\Lambda$ the element $c_1(L)$. Note that $l^2>0$. Let $\Lambda_{prim}$ be the orthogonal complement of $l$ in $\Lambda$, let $D$ be the period domain associated to $\Lambda_{prim}$, that is, 
$$D=\{x\in \mathbb P(\Lambda_{prim}\otimes\C)| x^2=0, x.\overline x>0\}.$$

Let $\widetilde O(\Lambda_{prim})$ be the group 
$$\widetilde O(\Lambda_{prim}):=\{g_{|\Lambda_{prim}}|g\in O(\Lambda), g(l)=l\}.$$
We will freely identify $\widetilde O(\Lambda_{prim})$ to a subgroup of $O(\Lambda)$ when needed.

Let $M$ be the monodromy group of $(X_0, L_0)$, see for instance \cite[Definition 1.1(5)]{Markman11}. Then $M$ can be identified with a subgroup of $\widetilde O(\Lambda_{prim})$. By a result of Sullivan \cite{Sullivan77}, $M$ has finite index in $\widetilde O(\Lambda_{prim})$ -- the result of Sullivan deals with the unpolarized case, see the discussion in \cite[Theorem 3.5]{Verbitsky13}, but the polarized case follows from \cite[Proposition 1.9]{Markman11}.

Let $\Gamma$ be a subgroup of finite index in both $M$ and a torsion-free arithmetic subgroup of $\widetilde O(\Lambda_{prim})$. By the theorem of Baily-Borel \cite{BailyBorel66}, the quotient $\Gamma\backslash D$ is a normal quasi-projective variety.

To any triple $(X, L, \phi)$ where $X$ is an irreducible holomorphic symplectic variety, $L$ is a line bundle on $L$ such that the pair $(X, L)$ is deformation-equivalent to $(X_0, L_0)$ and $\phi$ is an isomorphism $\phi : H^2(X, \Z)\ra \Lambda$ sending $c_1(L)$ to $l$, we can associate its \emph{period point} $\mathcal P(X, L, \phi)$. The element $\phi(H^{2, 0}(X))\subset \Lambda\otimes\C$ belongs to $D$, and we define $\mathcal P(X, L, \phi)$ to be the image of $\phi(H^{2, 0}(X))$ in $\Gamma\backslash D$. If $\gamma$ is any element of $\Gamma$, then $\mathcal P(X, L, \phi)=\mathcal P(X, L, \gamma\circ\phi)$.

Let $S$ be a smooth quasi-projective complex scheme, and let $\mathcal X\ra S$ be a smooth projective morphism whose fibers are irreducible holomorphic symplectic varieties. Let $L$ be a line bundle on $\mathcal X$ such that the pairs $(\mathcal X_s, L_s)$ are deformation-equivalent to $(X_0, L_0)$ for any complex point $s$ of $S$. Assume for simplicity that $S$ is connected and fix $s$ and such a deformation. Then, using parallel transport, we can identify $\Lambda$ and $H^2(\mathcal X_s, \Z)$. By definition of the monodromy group $M$, the monodromy representation $\rho : \pi_1(S, s)\ra O(\Lambda)$ factors through $M$. If $\rho$ factors through the finite index subgroup $\Gamma\subset M$, then the construction above induces a period map 
$$\mathcal P : S\ra \Gamma\backslash D.$$
By a result of Borel \cite{Borel69}, $\mathcal P$ is algebraic.

Let $(X, L, \phi)$ and $(X', L', \psi)$ be two triple as above. Assume that $\phi$ (resp. $\psi$) is induced by parallel transport along a deformation of $(X, L)$ (resp. $(X', L')$) to $(X_0, L_0)$. Then the global Torelli theorem of Verbitsky \cite{Verbitsky13} shows that if $\mathcal P(X, L, \phi)=\mathcal P(X', L', \psi)$, then $X$ and $X'$ are birational. In case $L$ and $L'$ are ample, this is the statement of \cite[Theorem 1.10]{Markman11}, and the general case can be deduced either by using a small deformation to the ample case or by using \cite[Proposition 1.9]{Markman11} to reduce to the general global Torelli theorem.

\bigskip

We claim that there exists $S, \mathcal X$ and $L$ as above such that the monodromy representation of each connected component of $S$ factors through $\Gamma$ and such that the image of $\mathcal P$ is $\Gamma\backslash D$. By noetherian induction, we can find $S, \mathcal X$ and $L$, as well as a Zariski open subset $U$ of $\Gamma\backslash D$, such that the image of the period map $\mathcal P : S\ra \Gamma\backslash D$ contains $U$ and such that $U$ is maximal with respect to this property. We assume by contradiction that $U$ is strictly contained in $\Gamma\backslash D$.

Let $Z$ be an irreducible component of the complement of $U$ in $\Gamma\backslash D$, and let $z$ be a very general complex point of $Z$. Using the surjectivity of the period map \cite[Theorem 8.1]{Huybrechts99}, we can find a triple $(X_z, L_z, \phi)$, where $(X_z, L_z)$ is deformation-equivalent to $(X_0, L_0)$ and $\phi : H^2(X_z, L_z)\ra \Lambda$ is induced by parallel transport, such that $\mathcal P(X_z, L_z)=z$.

By the aforementioned theorem of Huybrechts, $X_z$ is projective. Let $H_z$ be an ample line bundle on $X_z$. By the local Torelli theorem \cite[Th\'eor\`eme 5]{Beauville83}, the pair $(X_z, L_z)$ can be deformed over a small open subset of $Z(\C)$ -- for the usual topology. Since $z$ is a very general point of $Z$, the whole N\'eron-Severi group of $X_z$ deforms above this open subset, hence so does $H_z$. This shows that this deformation can be algebraized. 

As a consequence, resolving singularities of the base and passing to a finite cover, we can find a smooth projective morphism $\mathcal X_T\ra T$, where $T$ is a smooth complex quasi-projective variety whose fibers are irreducible holomorphic symplectic varieties, and $L$ a line bundle on $\mathcal X_T$ where the pairs $(\mathcal X_t, L_t)$ are deformation-equivalent to $(X_0, L_0)$ for any complex point $t$ of $T$, such that the modnodromy representation on $T$ factors through $\Gamma$ and the image of the period map 
$$\mathcal P : T\ra \Gamma\backslash D$$
is a Zariski-open subset $V$ of $Z$. Since $Z$ is an irreducible component of $(\Gamma\backslash D)\setminus U$, we can  shrink $V$ so that $V$ is open in $(\Gamma\backslash D)\setminus U$.

Now taking the disjoint union of the families $\mathcal X\ra S$ and $\mathcal X_T\ra T$, we get a family as above such that the image of the period map 
$$\mathcal P : S\sqcup T\ra \Gamma\backslash D$$
contains $U\cup V$, which is open in $\Gamma\backslash D$ and stricty contains $U$. This is the desired contradiction.

\bigskip 

Now let $S, \mathcal X$ and $L$ be as above such that the image of the period map is $\Gamma\backslash D$. By Lemma \ref{lemma:boundedness-big}, we can find integers $k, N$ and $d$ such that for any complex point $s$ of $S$, $h^0(\mathcal X_s, kL_s)\leq N$ and the complete linear system $|kL_s|$ induces a birational map from $\mathcal X_s$ to a subvariety of $\mathbb P|kL_s|$ of degree at most $d$.

Let $(X, L)$ be any pair as in the theorem that is deformation-equivalent to $(X_0, L_0)$. Let $\phi : H^2(X,\Z)\ra\Lambda$ be induced by parallell transport. By construction of $S$, we can find a complex point $s$ of $S$ such that $\mathcal P(X, L, \phi)=\mathcal P(s)$. As noted above, this implies by the global Torelli theorem that $X_s$ is birational to $X$. Since $X$ and $X_s$ have trivial canonical bundle, such a birational map is an isomorphism outside a closed subscheme of codimension at least $2$. In particular, it induces an isomorphism between the Picard groups of $X$ and $X_s$. Let $L'$ be the image of $L_s$ in the Picard group of $X_s$. Then $(X, L')$ satisfies the condition of the theorem.
\end{proof}

\begin{remark}
While using it simplifies slightly the phrasing of the proof, the global Torelli theorem of Verbitsky -- as well as the surjectivity of the period map -- could be replaced by the local Torelli theorem.
\end{remark}

\subsection{A variant of the Kuga-Satake construction and birational boundedness in positive characteristic}

The goal of this section is to extend part of the boundedness result above to positive characteristic. To facilitate the exposition, we will prove a weaker result. The proof is very similar to that of Theorem \ref{thm:Matsusaka-birational}, but we replaced the complex period map with the Kuga-Satake construction.

It is very likely that the construction by Pera in \cite{MadapusiPera13spin} of integral models of Shimura varieties of orthogonal type provides a period map that is sufficient to translate with only minor changes the proof of Theorem \ref{thm:Matsusaka-birational} to a positive characteristic setting -- which would also take care of the case of characteristic $3$. However, since one of the goals of this paper is to investigate the extent to which one can refrain from using too much of the theory of integral models of Shimura varieties, we decided to provide a slightly more elementary -- though certainly related -- proof. To simplify certain arguments, we will work in characteristic at least $5$. 

\bigskip

The Kuga-Satake construction associates an abelian variety to a polarized Hodge structure of weight $2$ with $h^{2,0}=1$. As shown by Deligne in \cite{Deligne72}, when applied to the primitive second cohomology group of a $K3$ surface, it is given by an absolute Hodge cycle. 

At least over the field of complex numbers, the Hodge-theoretic definition of the Kuga-Satake construction makes it possible to apply it to any irreducible holomorphic symplectic variety endowed with a line bundle such that $q(L)>0$, $q$ being the Beauville-Bogomolov form -- in that case, the orthogonal of $c_1(L)$ in $H^2(X, \Z)$ is indeed a polarized Hodge structure of weight $2$ with $h^{2,0}=1$. Most of the usual results on the arithmetic of the Kuga-Satake construction extend to this setting without any change in the proofs, as we explain in this section.

\bigskip

The following is the situation we will be considering.

\noindent\textbf{Setup.} Let $k$ be a perfect field of characteristic $p>3$, and let $W$ be the ring of Witt vectors of $k$. Let $K$ be the fraction field of $W$. Fix an embedding of $K$ into the field $\C$ of complex numbers. Let $T$ be a smooth, irreducible $W$-scheme, and let $\pi : \mathcal M\ra T$ be a smooth projective morphism. Let $L$ and $H$ be two line bundles on $\mathcal M$. We assume that $H$ is relatively ample. We fix a $k$-point $0$ of $T$.

We assume that $\mathcal M_\C\ra T_\C$ is a family of irreducible holomorphic symplectic manifolds, and that for any complex point $t$ of $T$, the restriction $L_t$ of $L$ to $\mathcal M_t$ satisfies $q(L_t)>0$, where $q$ is the Beauville-Bogomolov form. We assume that there is no torsion in the second and third singular cohomology groups of the fibers of $\mathcal M_\C\ra T_\C$. 

Let $\Lambda$ be a lattice isomorphic to $H^2(\mathcal M_t, \Z)$ for any complex point $t$ of $T$. Let $l$ and $h$ be elements of $\Lambda$ that are mapped to $c_1(L_t)$ and $c_1(H_t)$ under such an isomorphism. Let $\Lambda_l$, $\Lambda_h$ and $\Lambda_{l, h}$ be the orthogonal complement of $l$, $h$ and $\Z l +\Z h$ in $\Lambda$ respectively. We assume that the reduction modulo $p$ of the restriction of $q$ to $\Lambda_l$ is non-degenerate.

For the sake of later reference, we will turn the preceding situation into a definition. 

\begin{definition}
In the setup above, we say that the triple $(\mathcal M_0, H_0, L_0)$ is \emph{admissible} and that the lattice $\Lambda_{l}$ is a \emph{primitive lattice} for $(M_0, L_0)$. The quadratic form of the lattice is called the \emph{Beauville-Bogomolov quadratic form}. It induces a quadratic form on the N\'eron-Severi group of $\mathcal M_0$.

We say that $(\mathcal M_0, H_0, L_0)$ is \emph{strongly admissible} if in the setup above, we can ensure the following condition: let $\overline{\eta_k}$  be a geometric generic point of the special fiber $T_k$ of $T$ above $W$. Then $\mathcal M_{\overline{\eta_k}}$ is ordinary in degree $2$ -- that is, its second crystalline cohomology group has no torsion, and its Newton and Hodge polygons coincide --and the N\'eron-Severi group of $\mathcal M_{\overline{\eta_k}}$ is generated over $\Q$ by $c_1(L)$ and $c_1(H)$. 
\end{definition}

\begin{remark}
By \cite{DeligneIllusie87}, the Hodge to de Rham spectral sequence of $\mathcal M_0$ satisfies $E_1^{p,q}=E_{\infty}^{p,q}$ if $p+q=1$ or $p+q=2$. Furthermore, as in Proposition \ref{proposition:de-Rham-for-MH}, the hypotheses ensure that $h^{2,0}(\mathcal M_0)=h^{0,2}(\mathcal M_0)=1$.
\end{remark}

\bigskip 

The following result shows that moduli spaces of sheaves on $K3$ surfaces tend to be strongly admissible. We refer to \cite{Ogus79} for the definition of superspecial $K3$ surfaces. They consist of isolated points in the moduli space of polarized $K3$ surfaces and correspond to the singular locus of this moduli space.

\begin{proposition}\label{proposition:moduli-spaces-are-admissible}
In the situation of Theorem \ref{thm:Zarhins-trick}, assume that $k$ has characteristic at least $5$, and that $(X, H)$ is a polarized, non superspecial $K3$ surface over $k$. Then we can find a polarization $A$ on $\mathcal M_H(v)$ such that $(\mathcal M_H(v), A, L)$ is strongly admissible.
\end{proposition}

\begin{proof}
Dividing $H$ by an integer if necessary, we can assume that $H$ is primitive. Let $\widehat X\ra \widehat T$ be the formal universal deformation of the pair $(X, H)$. Since $X$ is not superspecial, $\widehat T$ is formally smooth of dimension $19$, i.e., $\widehat T$ is isomorphic to $\mathrm{Spf}\,W[[t_1, \ldots, t_{19}]]$. By Theorem \ref{thm:Zarhins-trick}, $c_1(v)$ is proportional to $H$, so $v$ lifts to $\widehat X$. Consider the relative moduli space $\mathcal M_H(\widehat X, v)$ over $\widehat T$. It is smooth and projective. As a consequence, we can find an ample line bundle $A$ on $\mathcal M_H(v)$ that lifts to $\mathcal M_H(\widehat X, v)$. Since $\mathcal M_H(\widehat X, v)\ra \widehat T$ is algebraizable, and by Proposition \ref{proposition:de-Rham-for-MH}, the only thing that remains to be proved to show that $(\mathcal M_H(v), A, L)$ is strongly admissible is that $L$ lifts to $\mathcal M_H(\widehat X, v)$. 

Let $P$ be the $W$-point of $\widehat T$ that corresponds to $t_1=\ldots=t_{19}=0$. By Theorem \ref{thm:Zarhins-trick}, we know that $L$ lifts to $\mathcal M_H(\widehat X, v)_P$. We prove by induction on $n$ that $L$ lifts to the $n$-th infinitesimal neighborhood of $P$ in $\widehat T$. Note that such liftings are unique since $H^1(\mathcal M_H(v), \mathcal O_{\mathcal M_H(v)})=0$. Furthermore, Proposition \ref{proposition:de-Rham-for-MH} also shows that the formation of $R^2\pi_*\mathcal O_{\mathcal M_H(\widehat X, v)}$ is compatible with base change.

We just showed that the result is true for $n=0$. Assume that $L$ lifts to the $n$-th infinitesimal neighborhood $P_n$ of $P$ in $\widehat T$. The obstruction to lifting $L$ to the $n+1$st infinitesimal neighborhood of $P$ belongs to $H^2(\mathcal M_H(\widehat X, v)_{P_n}, \mathcal O_{\mathcal M_H(\widehat X, v)_{P}})$. By Proposition \ref{proposition:de-Rham-for-MH} again, this group is a free $W$-module of rank $1$. However, Theorem \ref{thm:general-properties-of-moduli}, $(vi)$ shows that some power of $L$ actually lifts to $\widehat T$, so this obstruction is torsion. This shows that the obstruction vanishes and concludes the proof.
\end{proof}

\bigskip

We now investigate the Kuga-Satake construction in the setup above. If $\ell$ is a prime number different from $p$, we write $R^2\pi_*\Z_{\ell, prim}$ for the orthogonal of $c_1(L)$ in $R^2\pi_*\Z_\ell$. Let $n\geq 3$ be an integer prime to $p$. Up to replacing $k$ by a finite extension whose degree only depends on $n$ and the pair $(\Lambda, l)$, we can assume that the family $\mathcal M\ra T$ is endowed with a spin structure of level $n$ with respect to $R^2\pi_*\Z_{\ell, prim}$. We refer to \cite[3.2]{tateinv} and to \cite{Andre96, Rizov10,  Maulik12} for definitions and details. 

\bigskip

Let $\mathcal S_{n, l, h}$, $\mathcal S_{n, h}$ and $\mathcal S_{n, l}$ be the orthogonal Shimura varieties with spin level $n$ associated to $\Lambda_{l, h}$, $\Lambda_h$ and $\Lambda_{l}$ respectively, see \cite[3.6]{tateinv}. Then these three varieties are all defined over $\Q$ and we have closed embeddings of $\mathcal S_{n, l, h}$ into $\mathcal S_{n, h}$ and $\mathcal S_{n, l}$. These are both defined over $\Q$.

The period map $\mathcal P$, as defined for instance in the previous section, gives a morphism 
$$\mathcal P : T_\C\ra \mathcal S_{n, l, h}.$$
The argument of \cite[Proposition 16]{tateinv}, which is essentially contained in \cite[Appendix 1]{Andre96}, the composition 
\[
\xymatrix{T_{\C}\ar[r]^{\mathcal P} & \mathcal S_{n, l, h}\ar[r] & \mathcal S_{n, h}
}
\]
is defined over $K$. Since the second map is a closed immersion defined over $\Q$, this shows that 
$$\mathcal P : T_\C\ra \mathcal S_{n, l, h}$$
is defined over $K$, and so is 
$$\mathcal P_l : T_\C \ra \mathcal S_{n, l}$$
defined as the composition with $\mathcal S_{n, l, h}\ra \mathcal S_{n, l}$.

The Kuga-Satake construction induces a morphism 
\begin{equation}\label{equation:Kuga-Satake-Shimura}
KS : \mathcal S_{n, l}\ra \mathcal A_{g, d', n, \Q}
\end{equation}
where $\mathcal A_{g, d', n, \Q}$ is the moduli space over $\Q$ of abelian varieties of dimension $g$ with a polarization of degree $d'^2$ and level $n$ structure, for some integers $g$ and $d'$, where $d'$ is prime to $p$, see \cite{Andre96}. Let
$$\kappa_K : T_K\ra \mathcal A_{g, d', n, K}$$
be the composition of $KS$ with $\mathcal P$.

Let $\psi : \mathcal A_K\ra T_K$ be the abelian scheme induced by $\kappa_K$, and let $C=C(\Lambda_l)$ be the Clifford algebra associated to $\Lambda_l$. Then there is a canonical injection of $C$ into the ring of endomorphisms of the scheme $\mathcal A_K$ and, as shown in \cite[6.5]{Deligne72}, we have an isomorphism of $\ell$-adic sheaves of algebras on $T_K$
\begin{equation}\label{equation:l-adic-KS}
C(R^2\pi_*\Z_\ell(1)_{prim})\simeq \mathrm{End}_C(R^1\psi_*\Z_\ell)
\end{equation}
where $C$ denotes the Clifford algebra -- here we are using the Kuga-Satake construction with respect to the full Clifford algebra.

By \cite[6.1.2]{Rizov10}, the data above is sufficient to show that the Kuga-Satake morphism $\kappa_K$ can be extended in a unique way to a morphism over $W$
$$\kappa : T\ra \mathcal A_{g, d', n}$$
where $\mathcal A_{g, d', n}$ denotes the moduli space over $W$. 

\begin{definition}\label{definition:Kuga-Satake-mapping}
The morphism 
$$\kappa : T\ra \mathcal A_{g, d', n}$$
is the Kuga-Satake mapping.
\end{definition}

\bigskip

We know recall some properties of the Kuga-Satake construction. Assume that $k$ is algebraically closed. Recall that $0$ is a $k$-point of $T$, and let $\widehat T$ is the formal neighborhood of $0$ in $T$. As in \cite[Section 6]{Maulik12}, and by the argument of \cite[Proposition 13]{tateinv}, we have the following canonical primitive strict embedding of filtered Frobenius crystals 
\begin{equation}\label{equation:crystalline-KS}
R^2\pi_*\Omega^{\bullet}_{\widehat{\mathcal M}/\widehat T}(1)_{prim}\hookrightarrow \mathrm{End}_C(R^1\psi_*\Omega^{\bullet}_{\widehat{\mathcal A} /\widehat T}).
\end{equation}
It is compatible with (\ref{equation:l-adic-KS}) and the Beauville-Bogomolov form via the comparison theorems. In particular, we get a primitive isometry
\begin{equation}\label{equation:special-classes}
H^2_{cris}(\mathcal M_0/W)\hookrightarrow \mathrm{End}(H^1_{cris}(\mathcal A_0/W))
\end{equation}

\begin{lemma}
Let $\Lambda_{l}$ be a lattice, and let $A_0$ be an abelian variety of dimension $g$ over $k$, together with a level $n$ structure and a polarization of degree $d'$. Then there are only finitely many subspaces $V\subset \mathrm{End}(H^1_{cris}(\mathcal A_0/W))$ that arise as the image of some $H^2_{cris}(\mathcal M_0/W)$ for some admissible triple $(\mathcal M_0, H_0, L_0)$ with primitive lattice $\Lambda_{l}$.
\end{lemma}

\begin{proof}
By the main construction and result of \cite{Kisin10}, and since $l^2$ is not divisible by $p$, the Shimura variety $\mathcal S_{n, l}$ admits a smooth canonical integral model $\overline{\mathcal S_{n, l}}$ over $W$, and the Kuga-Satake morphism $KS : \mathcal S_{n, l}\ra \mathcal A_{g, d', n, \Q}$ extends to a finite, unramified morphism
\begin{equation}\label{equation:KS-integral-model}
\overline{KS} : \overline{\mathcal S_{n, l}}\ra \mathcal A_{g, d', n}.
\end{equation}

Let $P$ be the canonical locally $\mathcal O_{\mathcal S_{n, l}}$-module endowed with a connection $\nabla$ and a Hodge filtration of weight $0$ on $\mathcal S_{n, l}$ over $K$. Denote again by $\psi : \mathcal A\ra \overline{S_{n, l}}$ the abelian scheme induced by $\overline{KS}$. Then by definition of the Kuga-Satake construction we have a morphism over $\mathcal S_{n, l}$
\begin{equation}\label{equation:KS-dR}
P\hookrightarrow \mathrm{End}(R^1\psi_*\Omega^{\bullet}_{\mathcal A/\mathcal S_{n, l}})
\end{equation}
analogous to (\ref{equation:crystalline-KS}). It is compatible with the Hodge filtrations and the connexions on both sides.

Now let $\mathcal A_0$ be as in the lemma. It suffices to show that there are a finite number of subspaces $V'\subset \mathrm{End}(H^1_{cris}(\mathcal A_0/W)\otimes K)$ that arise as stated. Such a $V'$ is obtained by first picking a preimage of the point of $\mathcal A_{g, d', n}$ corresponding to $\mathcal A_0$, then lifting this preimage to a $W$-point of $\overline{\mathcal S_{n, l}}$. If $A_K$ is the corresponding abelian variety over $K$, then the relation (\ref{equation:KS-dR}) induces a subspace of $\mathrm{End}(H^1_{dR}(\mathcal A_K/K))$. Via the comparison theorem between the de Rham cohomology of $\mathcal A_K$ and the crystalline cohomology of $\mathcal A_0$, this induces the subspace $V'$ of $\mathrm{End}(H^1_{cris}(\mathcal A_0/W)\otimes K)$. 

Since (\ref{equation:KS-dR}) is compatible with the Gauss-Manin connection, it is readily seen that the construction above only depends on the choice of a preimage of the point corresponding to $\mathcal A_0$ in $\overline{\mathcal S_{n, l}}$ under $\overline{KS}$. Since $\overline{KS}$ is finite, this shows the result.
\end{proof}

\begin{remark}
The proof of the lemma above is the only appearance in the text of canonical integral models of Shimura varieties. While they make the proof more natural, we do not really make full use of their properties, and they could be replaced by any model that allows us to extend the Kuga-Satake morphism over $W$.
\end{remark}

\begin{remark}
The idea of the proof could be extended to show that there is only finitely many birational equivalence classes of varieties $\mathcal M_0$ as above that have a given Kuga-Satake variety. We will prove a weaker result instead.
\end{remark}

The following lemma appears in the proof of \cite[Proposition 22]{tateinv}. For $K3$ surfaces, it is stated and proved in \cite[Proposition 4.17 (4)]{MadapusiPera13}, see also \cite[Proposition 2.3]{Benoist14}. 

\begin{lemma}\label{lemma:comparison-Picard-groups}
Assume that $k$ is algebraically closed. Let $(\mathcal M_0, H_0, L_0)$ be a strongly admissible triple over $k$, and let $A_0$ be the abelian variety associated to $\mathcal M_0$ under the Kuga-Satake mapping. Let $i : H^2_{cris}(\mathcal M_0/W)\hookrightarrow \mathrm{End}(H^1_{cris}(\mathcal A_0/W))$ be the morphism (\ref{equation:special-classes}).

Then we have an isomorphism of lattices 
$$\{\alpha\in NS(\mathcal M_0)|q(\alpha, c_1(L_0))=0\}\simeq \mathrm{End}(A_0)\cap \mathrm{Im}(i).$$
\end{lemma}

\begin{proof}
The arguments of the proof of \cite[Proposition 22]{tateinv} -- which are essentially the same as those of \cite[Proposition 4.17 (4)]{MadapusiPera13} -- apply without any change. For the sake of completeness, we briefly recall the argument. We work in the setup above.

Let $\alpha$ be a class in the N\'eron-Severi group of $\mathcal M_0$ that is orthogonal to $c_1(L_0)$. Since $h^{0,2}(\mathcal M_0)=1$, the assumption on the geometric generic fiber of $T$ ensures that $\alpha$ lifts to characteristic zero. In particular, it lifts to a Hodge class. Via the Kuga-Satake correspondence and the Hodge conjecture for endomorphisms of abelian varieties, it induces an endomorphism of a lift of $A_0$, hence of $A_0$ itself. This endomorphism belongs to the image of $i$ by construction. 

Now let $\beta$ be a class in $\mathrm{End}(A_0)\cap \mathrm{Im}(i)$. The argument of \cite[End of Proposition 22]{tateinv}, which is a rephrasing of \cite[Theorem 2.9]{Ogus79}, shows that there exists a lift of $\mathcal M_0$ parametrized by $T$ such that $\beta$ lifts to an endomorphism of the induced Kuga-Satake abelian variety. In particular, it lifts to a Hodge class, which by the Hodge conjecture for divisors allows us to conclude as before that $\beta$ was induced by a line bundle on $\mathcal M_0$, orthogonal to $c_1(L)$.
\end{proof}

The two results above directly imply the following finiteness result, which is a weak version of Theorem \ref{thm:Matsusaka-birational}.

\begin{proposition}\label{proposition:weak-bb}
Let $k$ be a finite field of characteristic at least $5$. Let $r$ and $n$ be two positive integers. Then there exist finitely lattices $\Lambda_1, \ldots, \Lambda_r$ such that if $(\mathcal M_0, H_0, L_0)$ is a strongly admissible triple over $k$ with $\mathrm{dim}(X)=2n$ and $c_1(L_0)^{2n}=r$, then 
$$NS(\mathcal M_{0, \overline k})\simeq \Lambda_i$$
for some integer $i$., where the left-hand side is endowed with the Beauville-Bogomolov form.
\end{proposition}

\begin{proof}
Since by assumption $(\mathcal M_0, H_0, L_0)$ lifts to $\C$, we can apply Lemma \ref{lemma:finiteness-deformations} to show that there exist finitely many lattices that can appear as a primitive lattice for $(\mathcal M_0, L_0)$ and that $q(L_0)$ is uniformly bounded. As a consequence, we can restrict our attention to those sternly admissible triple with primitive lattice $\Lambda_{l}$ for some fixed $\Lambda_{l}$.

Fix some integer $n\leq 3$. After replacing $k$ by a finite extension whose degree only depends on $n$ and $\Lambda_{l}$ so that spin structures are defined on suitable deformations of $(\mathcal M_0, H_0, L_0)$ as before, we can construct the Kuga-Satake abelian variety $A_0$ together with the canonical morphism 
$$i : H^2_{cris}(\mathcal M_0/W)\hookrightarrow \mathrm{End}(H^1_{cris}(\mathcal A_0/W)).$$
If $d'$ is as in Definition \ref{definition:Kuga-Satake-mapping}, then $A_0$ is a polarized abelian variety over $k$ of degree $d'$, together with a level $n$ structure. Since $k$ is finite, there are only finitely many such $A_0$. Furthermore, given $A_0$, there are only finitely many subspaces $V\subset \mathrm{End}(H^1_{cris}(A_0/W))$ that arise as the image of a morphism $i$ as above. 

By Lemma \ref{lemma:comparison-Picard-groups}, this discussion shows that the lattice $c_1(L_0)^{\perp}:=\{\alpha\in NS(\mathcal M_{0, \overline k})|q(\alpha, c_1(L_0))=0\}$ can take only finitely many values as the strongly admissible triple varies. The inequality 
$$|\mathrm{disc}(NS(\mathcal M_{0, \overline k}))|\leq q(L_0)|\mathrm{disc}(c_1(L_0)^{\perp})|$$
shows that the discriminant of $NS(\mathcal M_{0, \overline k})$ is bounded. Since the set of isomorphism classes of lattices with bounded rank and discriminant is finite by \cite[Chap. 9, Theorem 1.1]{Cassels78}, this shows the result.
\end{proof}

\subsection{Finiteness results for $K3$ surfaces over finite fields}

The following weak finiteness result for $K3$ surfaces over finite fields will be the key to the proof of the Tate conjecture for $K3$ surfaces. We refer to \cite{Ogus79} for the definition of superspecial $K3$ surfaces. They consist of isolated points in the moduli space of polarized $K3$ surfaces and correspond to the singular locus of this moduli space.

\begin{proposition}\label{prop:weak-finiteness}
Let $k$ be a finite field of characteristic at least $5$. Let $\overline k$ be an algebraic closure of $k$ and let $W$ be the ring of Witt vectors of $\overline k$. 

Let $d$ and $t_0$ be positive integers. Then there exists a positive integer $N$ and nonzero integers $a,b$ such that there exist only finitely many polarized non-superspecial $K3$ surfaces $(X, H)$ of degree $2md$ over $k$ , where $m$ is a positive integer satisfying 
\begin{enumerate}[(i)]
\item $m=1[N]$,
\item $m$ is prime to $a$ and $b$, and both $a$ and $b$ are quadratic residues modulo $m$.
\item The $p$-adic valuation of the discriminant of $NS(X_{\overline k})$ is at most $t_0$.
\end{enumerate}
\end{proposition}

\begin{proof}
We fix integers $r, n, N, a, b$ as in Theorem \ref{thm:Zarhins-trick}. As a consequence, if $(X, H)$ is a polarized $K3$ surface over $k$ as above, we can find a Mukai vector $v$ on $X$ satisfying condition $(C)$ of Definition \ref{def:fine-moduli-space} such that the moduli space $\mathcal M_H(v)$ has dimension $4$ and there exists a line bundle $L$ on $\mathcal M_H(v)$ satisfying $c_1(L)^{4}=r$ and $q(L)>0$. Proposition \ref{proposition:moduli-spaces-are-admissible}, there exists an ample line bundle $A$ on $\mathcal M_H(v)$ such that $(\mathcal M_H(v), A, L)$ is strongly admissible.

As a consequence of Proposition \ref{proposition:weak-bb}, we can find finitely many lattices $\Lambda_1, \ldots, \Lambda_s$, depending only on $d, N, a, b$, such that if $X, H$ and $v$ are as above, then
$$NS(\mathcal M_H(v)_{\overline k})\simeq \Lambda_i.$$

Let $p$ be the characteristic of $k$. By Corollary \ref{corollary:comparison-of-discriminants}, we can write 
$$|\mathrm{disc}(NS(X_{\overline k}))| = p^t\lambda |\mathrm{disc}(\Lambda_i)|$$
for some $\lambda\leq v^2=n-2$ and some nonnegative integer $t$. Since the $p$-adic valuation of $\mathrm{disc}(NS(X_{\overline k}))$ is bounded by assumption, this shows that the discriminant of $NS(X_{\overline k})$ is bounded independently of $X$. Since the set of lattices with bounded rank and discriminant is finite by \cite[Chap. 9, Theorem 1.1]{Cassels78}, Corollary \ref{corollary:basic-finiteness} shows that the set of isomorphism classes of $K3$ surfaces $X$ as in the statement is finite.
\end{proof}

\section{The Tate conjecture for $K3$ surfaces over finite fields}\label{section:twisted}

This section is devoted to the application of moduli spaces of twisted sheaves to the Tate conjecture for $K3$ surfaces. This circle of ideas originates in \cite{LieblichMaulikSnowden14}, though it is in the line of \cite{ArtinSwinnertonDyer73}. As explained \cite{LieblichMaulikSnowden14}, the failure of the Tate conjecture for a given $K3$ surface $X$ over a finite field can be rephrased as the existence of an infinite family of so-called twisted Fourier-Mukai partners of $X$ over $k$. Finiteness results as we proved above allow us to control such families, thus proving Theorem \ref{thm:main} and Theorem \ref{thm:picard-number-2}.

We will need a variation on the techniques of \cite{LieblichMaulikSnowden14} to adapt their results to a slightly more flexible setting and make it work in arbitrary characteristic. After recalling basic facts on moduli spaces of twisted sheaves, we give a very short proof of Theorem \ref{thm:picard-number-2} and prove Theorem \ref{thm:main}.

\subsection{Moduli spaces on twisted sheaves on $K3$ surfaces}

We briefly recall the theory of moduli spaces of twisted sheaves on a $K3$ surface. We refer to the discussion in \cite[3.1 to 3.4]{LieblichMaulikSnowden14} for details. 

Let $X$ be a $K3$ surface over a field $k$. Let $\ell$ be a prime number invertible in $k$. An \emph{$\ell$-adic $B$-field} on $X$ is an element 
$$B=\alpha/\ell^n\in H^2(X, \Q_\ell(1))$$
where $\alpha\in H^2(X, \Z_\ell(1))$ is primitive. The Brauer class associated to $B$ is the image of $\alpha$ under the composition
$$H^2(X, \Z_\ell(1))\ra H^2(X, \mu_{\ell^n})\ra Br(X)[\ell^n].$$
It is denoted by $[\alpha_n]$. 

Let $B=\alpha/\ell^n$ be an $\ell$-adic $B$-field on $X$ and write $r=\ell^n$. Following \cite[(3.4)]{Yoshioka06}, we define 
$$T_{-\alpha/r} : \widetilde H((X_{\overline k}, \Z_\ell)\ra \widetilde H((X_{\overline k}, \Q_\ell), x\mapsto x\cup e^{-\alpha/r}.$$

Let $N^{\alpha/r}(X)$ be the preimage of $N(X)$ by $T_{-\alpha/r}$. This is the group denoted by $CH^{\alpha/r}(X, \Z)$ in \cite{LieblichMaulikSnowden14}. By \cite[Lemma 3.3.3]{LieblichMaulikSnowden14}, we have 
\begin{equation}\label{eqn:Mukai-twisted}
N^{\alpha/r}(X)=\{(ar, D+a\alpha, c\omega)|(a, c\in \Z, D\in NS(X)\}\subset \widetilde H((X_{\overline k}, \Z_\ell).
\end{equation}
Elements of $N^{\alpha/r}(X)$ are called \emph{twisted Mukai vectors} on $X$. 

\bigskip

Let $\mathcal X\ra X$ be a $\mu_r$-gerbe representing the class $[\alpha_n]$. Given an $\mathcal X$-twisted sheaf on $X$, we can define its Mukai vector as an element of $N^{\alpha/r}(X)$. Let $v$ be a primitive element in $N^{\alpha/r}(X)$. Assume that $\mathrm{rk}(v)=r$ and $v^2=0$. Then by \cite[Proposition 3.4.1]{LieblichMaulikSnowden14}, the stack of simple $\mathcal X$-twisted sheaves on $X$ with Mukai vector $v$ is a $\mu_r$-gerbe over a $K3$ surface $\mathcal M(v)$ -- denoted by $M_{\mathcal X}(v)$ in \cite{LieblichMaulikSnowden14}.

The discriminant of $N^{\alpha/r}(X_{\overline k})$ is easy to compute.

\begin{lemma}\label{lemma:twisted-discriminant}
With the notations above, we have
$$|\disc(N^{\alpha/r}(X_{\overline k}))|=r^{2}|\disc(NS(X_{\overline k})|$$
\end{lemma}

\begin{proof}
Let $D_1, \ldots, D_s$ be a basis of the free $\Z$-module $NS(X_{\overline k})$. Then by (\ref{eqn:Mukai-twisted}), a basis of $N^{\alpha/r}(X_{\overline k}))$ is given by 
$$(r, \alpha, 0), (0, 0, 1), (0, D_1, 0), \ldots, (0, D_s, 0).$$
The result follows immediately.
\end{proof}

We now relate the N\'eron-Severi group of a $2$-dimensional moduli space of twisted sheaves on a $K3$ surface with that of the N\'eron-Severi group of the $K3$ surface. The following discussion parallels Theorem \ref{thm:general-properties-of-moduli}. We do not repeat the arguments allowing us to deduce results over arbitrary fields from the results over the field of complex numbers.

Let $\ell$ be a prime number that is invertible in $k$. Theorem 3.19, $(ii)$ in \cite{Yoshioka06} shows that there exists a canonical bijective isometry
$$v^{\perp, \ell}/\Z_\ell v\ra H^2(\mathcal M(v)_{\overline k}, \Z_\ell(1))$$
induced by an algebraic correspondence, where $v^{\perp, \ell}$ is the orthogonal of $v$ in the $\ell$-adic Mukai lattice of $X$. Note that $v$ is isotropic by assumption, so $\Z_\ell v\subset v^{\perp, \ell}$.

The exact same argument as in the proof of \ref{thm:general-properties-of-moduli}, $(vi)$ and $(vii)$ shows that -- if $k$ is algebraically closed or finite -- there exists an injective isometry
\begin{equation}\label{eq:twisted-NS}
\theta_v : v^{\perp}/\Z v\ra NS(\mathcal M(v)),
\end{equation}
where $v^{\perp}$ is the orthogonal of $v$ in the lattice of twisted Mukai vectors on $X_{\overline k}$, $N^{\alpha/r}(X)$. Furthermore, the cokernel of $\theta_v$ is a $p$-primary torsion group, where $p$ is the characteristic of $k$.

\begin{proposition}\label{proposition:discriminant-of-twisted-partners}
With the notations above, let $n_v$ be the positive integer defined by
$$v.N^{\alpha/r}(X_{\overline k})=n_v\Z.$$
If $k$ has positive characteristic $p$, then there exists a nonnegative integer $t$ such that 
$$(n_v)^2p^t\disc(NS(\mathcal M(v))_{\overline k})=r^2\disc(NS(X_{\overline k})).$$
If $k$ has characteristic zero, then 
$$(n_v)^2\disc(NS(\mathcal M(v))_{\overline k})=r^2\disc(NS(X_{\overline k})).$$
\end{proposition}

\begin{proof}
By (\ref{eq:twisted-NS}), and as in Corollary \ref{corollary:comparison-of-discriminants}, we can find a nonnegative integer $t$ such that 
$$|\disc(v^{\perp}/\Z v)|=p^t |\disc(NS(\mathcal M(v)_{\overline k}))|$$
We now relate the discriminant of $v^{\perp}/\Z v$ to that of $N^{\alpha/r}(X_{\overline k})$. Let $e_1, \ldots, e_t$ be elements of $v^{\perp}$ that form a basis of $v^{perp}/\Z v$. Then $v, e_1, \ldots, e_t$ is a basis of $v^{\perp}$. Let $w$ be an element of $N^{\alpha/r}(X_{\overline k})$ such that $v.w=n_v$. Then $w, v, e_1, \ldots, e_t$ is a basis of $N^{\alpha/r}(X_{\overline k})$. Computing the discriminant of $N^{\alpha/r}(X_{\overline k})$ in this basis, we get 
$$|\disc(N^{\alpha/r}(X_{\overline k}))|=(n_v)^2|\disc(v^{\perp}/\Z v)|.$$
Using Lemma \ref{lemma:twisted-discriminant}, we finally get
$$r^2|\disc(NS(X_{\overline k})|=(n_v)^2p^t|\disc(NS(\mathcal M(v)_{\overline k}))|.$$
\end{proof}

\subsection{Finiteness statements and the Tate conjecture for $K3$ surfaces}

The goal of this section is to prove Theorem \ref{thm:main}. In \cite{LieblichMaulikSnowden14}, the authors prove the following statement. 

\begin{theorem}\label{theorem:Maulik-Lieblich-Snowden}
Let $k$ be a finite field of characteristic at least $5$. Assume that there are only finitely many $K3$ surfaces defined over each finite extension of $k$. Then the Tate conjecture holds for all $K3$ surfaces over $\overline k$.
\end{theorem}

We will not be able to use the theorem above directly, and will rely on a simplified version of the argument of \cite{LieblichMaulikSnowden14}, which holds in arbitrary characteristic.

From now on, let $X$ be a $K3$ surface over a finite field $k$ of characteristic $p$. Up to replacing $k$ by a finite extension, we can and will assume that $NS(X_{\overline k})=NS(X)$.
 
If $\ell$ is a prime number different from $p$, we denote by $T(X, \Z_{\ell})$ the orthogonal complement of $NS(X)\otimes\Z_{\ell}$ in $H^2(X, \Z_\ell(1))$. Note that by the Hodge index theorem for surfaces, the intersection form on $NS(X)$ is non-degenerate.

The following is a slightly modified version of Lemma 3.5.1 of \cite{LieblichMaulikSnowden14} that gets rid of the hypothesis on the characteristic of $k$ and does not make use of \cite{ArtinSwinnertonDyer73}.
\begin{lemma}\label{lemma:new-lemma-3.5.1}
Assume that $Br(X)$ is infinite. There exist prime numbers $p_1, \ldots, p_r$ such that if $\ell$ is big enough and $p_i$ is a square modulo $\ell$ for all $i$, there exists $\alpha\in T(X, \Z_{\ell})$ such that $\alpha^2=1$.
\end{lemma}

\begin{proof}
By Proposition 2.1.2 of \cite{LieblichMaulikSnowden14}, we can assume that the discriminant of $H^2(X, \Z_{\ell}(1))$ has $\ell$-adic valuation zero.  As a consequence, again if $\ell$ is big enough, the discriminant of $T(X, \Z_{\ell})$ has $\ell$-adic valuation zero. Furthermore, there are only finitely many prime numbers $p_i$ such that the discriminant of $T(X, \Z_{\ell})$ can have nonzero $p_i$-adic valuation as $\ell$-varies by Proposition 2.1.2 again. 

By Hensel's lemma, we can prove the result after tensoring $T(X, \Z_{\ell})$ with $\mathbb F_{\ell}$. If the rank of $T(X, \Z_{\ell})$ is at least $2$, $T(X, \Z_{\ell})\otimes\mathbb F_{\ell}$ represents $1$ by general results. If the rank is $1$, the result holds since its discriminant is a square by assumption.
\end{proof}

For the sake of reference, we also state the following easy lemma.

\begin{lemma}\label{lemma:existence-of-quadratic-residues}
Let $x_1, \ldots, x_r$ be finitely many integers. Then there exists infinitely many prime number $\ell$ such that all the $x_i$ are quadratic residues modulo $\ell$.
\end{lemma}

\begin{proof}
We can assume that $x_1=-1, x_2=2$ and the remaining $x_i$ are distinct odd prime numbers. Then choosing $\ell$ to be congruent to $1$ modulo $8$ ensures that $x_1$ and $x_2$ are quadratic residues modulo $\ell$. Using the quadratic reciprocity law and Dirichlet's theorem on primes in arithmetic progressions allows us to conclude.
\end{proof}

Now let $D$ be a line bundle of degree $2d$ on $X$, and let $t_0$ be the $p$-adic valuation of $\disc(NS(X))$. Let $N, a$ and $b$ be as in Proposition \ref{prop:weak-finiteness}. Let $\ell$ be a big enough prime number, different from $p$, such that $2, a, b, -2d$ are quadratic residues modulo $\ell$. Assume that $\ell$ satisfies the condition of Lemma \ref{lemma:new-lemma-3.5.1}. This is possible by Lemma \ref{lemma:existence-of-quadratic-residues}.

Assume that $X$ does not satisfy the Tate conjecture, so that $Br(X)$ is infinite by \cite[Proposition 3.2.8]{LieblichMaulikSnowden14} for instance. As in \cite[Proposition 3.5.4]{LieblichMaulikSnowden14}, the assumptions on $\ell$ and Lemma \ref{lemma:new-lemma-3.5.1} show that we can find $\gamma\in T(X, \Z_\ell)$ such that $\gamma^2=-2d$. 

If $n$ is a positive integer, let 
$$v_n=(\ell^n, \gamma+D, 0)\in N^{\gamma/\ell^n}(X).$$
As in \cite[Proposition 3.5.4]{LieblichMaulikSnowden14}, we have $v_n^2=0$. As in section \ref{section:twisted}, the twisted moduli space $X_n :=\mathcal M_{H_n}(v_n)$ is a $K3$ surface over $k$. 

By Proposition \ref{proposition:discriminant-of-twisted-partners}, 
$$\lambda_n^2 p^t\disc(NS((X_n)_{\overline k}))=\ell^{2n}\disc(NS(X_{\overline k})).$$
Here $\lambda_n$ is an integer such that 
$$v_n.N^{\gamma/\ell^n}=\lambda_n\Z.$$
Since $v_n.(0, D, 0)=-2d$, we have $\lambda_n^2\leq 4d^2$, so that the $\ell$-adic valuation of $\disc(NS((X_n)_{\overline k}))$ goes to infinity as $n$ goes to infinity. 

We now use the surfaces $X_n$ to prove Theorems \ref{thm:main} and \ref{thm:picard-number-2}, i.e. that $X$ satisfies the conjecture if the characteristic is at least $5$ or if the Picard number is at least $2$. The second one does not rely on Theorem \ref{thm:Zarhins-trick} and only uses Proposition \ref{proposition:birational-boundedness-for-K3}. It should be seen as a modern rephrasing of \cite{ArtinSwinnertonDyer73}.

\begin{proof}[Proof of Theorem \ref{thm:picard-number-2}]
Assume that $X$ has Picard number at least $2$. Then we can assume that $d$ is negative and find a divisor $B$ on $X$, orthogonal to $D$, such that $B^2=2e>0$. 

Let $b_n=(0, B, 0)\in N^{\gamma/\ell^n}(X)$. Then $b_n.v_n=0$ and $b_n^2=2e$. By equation (\ref{eqn:Mukai-twisted}), this shows that there exists a divisor $B_n$ on $X$ with $B_n^2=2e>0$. Corollary \ref{corollary:basic-finiteness} implies that the surfaces $X_{n, \overline k}$ fall into finitely many isomorphism classes, which is in contradiction with the fact that the $\ell$-adic valuation of $\disc(NS((X_n)_{\overline k}))$ goes to infinity as $n$ goes to infinity. 
\end{proof}

\begin{proof}[Proof of Theorem \ref{thm:main}]
Assume that the characteristic of $k$ is at least $5$. We can assume that $X$ is not superspecial. Indeed, these admit supersingular deformations that are not superspecial by \cite[Remark 2.7]{Ogus79}, so the Tate conjecture for these follow from the result for non superspecial surfaces and \cite[Theorem 1.1]{Artin74}.

We assume that $D$ is ample. By equation (\ref{eqn:Mukai-twisted}), we have
$$h_n:=(\ell^{2n}, \ell^n\gamma, -2d)\in N^{\gamma/\ell^n}(X_{\overline k}).$$

Furthermore, we have 
$$h_n.v_n=\ell^n\gamma^2+2d\ell^n=0$$ 
and
$$h_n^2=\ell^{2n}\gamma^2+4d\ell^{2n}=2d\ell^{2n}.$$
By equation (\ref{eq:twisted-NS}), this shows that there exists a line bundle $H_n$ on $X_n$ with $c_1(B)^2=2d\ell^{2n}$. It is easy to show that $H_n$ is ample. Indeed, if $\mathcal X\ra S$ is a deformation of $(X, H)$ over a discrete valuation ring with generic fiber of Picard rank $1$, then $\mathcal M(v)$ lifts to a projective scheme over $S$ with generic fiber of Picard rank $1$, generated by a lift of $H_n$.

Using again the equality 
$$\lambda_n^2 p^t\disc(NS((X_n)_{\overline k}))=r^2\disc(NS(X_{\overline k})).$$
with $\lambda_n^2\leq 4d^2$, we get that the $p$-adic valuation of $\disc(NS((X_n)_{\overline k}))$ is bounded independently of $n$.

We now restrict to the positive integers $n$ such that $\ell^{2n}=1[N]$. By Proposition \ref{prop:weak-finiteness}, there exist only finitely many $K3$ surfaces over $k$ satisfying the condition above, which is in contradiction with the fact that the $\ell$-adic valuation of $\disc(NS((X_n)_{\overline k}))$ goes to infinity as $n$ goes to infinity and proves that $X$ satisfies the Tate conjecture.

\end{proof}

\bibliographystyle{alpha}
\bibliography{ZarhinK3}

\end{document}